\documentclass[journal]{IEEEtran}
\usepackage[english]{babel}
\usepackage[usenames,dvipsnames]{xcolor} % introduce colors like OliveGreen
\usepackage{amsmath}
\usepackage{amssymb}
\usepackage{amsthm}
\usepackage{amsmath}
\usepackage{mathtools} %mathclap
\usepackage{graphicx}
\usepackage{url}
\usepackage{float}
\usepackage[hidelinks]{hyperref}
\usepackage[utf8]{inputenc}
\usepackage{dsfont}             % for nice \1
\usepackage[normalem]{ulem} % strike out text, example:  \sout{Hello World}

\newcommand{\w}{\omega}
\newcommand{\1}{\mathds 1  }
\newcommand{\comm}{c}
\newcommand{\p}{\partial}

\newcommand{\pH}{port-Hamiltonian}
\newcommand{\PH}{Port-Hamiltonian}
\newcommand{\G}{\mathcal{G} }
\newcommand{\V}{\mathcal{V} }

\newcommand{\E}{\mathcal{E} }

\newcommand{\Sin}{\,\textbf{sin}\,}
\newcommand{\alert}[1]{{\color{red}\sout{#1}}}  % red, sout
\newcommand{\blue}[1]{{\color{blue}#1}}         % blue
\definecolor{new}{rgb}{0.55,0,0.55}
\newcommand{\new}[1]{{\color{blue}#1}}         % blue
\newcommand{\eqset}{{\color{blue}\mathcal Z}}

\renewcommand{\alert}[1]{}                      % use this to ignore red
\renewcommand{\blue}[1]{{\color{black}#1}}      % use this to ignore blue
\renewcommand{\new}[1]{{\color{black}#1}}       % use this to ignore purple
\renewcommand{\eqset}{\mathcal Z}               % use this to ignore \eqset color

\DeclareMathOperator{\col}{col}

\DeclareMathOperator{\diag}{diag}

\newtheorem{mythm}{Theorem} %[section]
\newtheorem{myprop}{Proposition} %[section]
\theoremstyle{remark}
\newtheorem{myrem}{Remark} %[section]
\theoremstyle{definition}
\newtheorem{myass}{Assumption} %[section]

\ifCLASSINFOpdf
\else
\fi
\begin{document}
\title{A unifying energy-based approach to stability of power grids with market dynamics}
% A unifying energy-based approach to optimal frequency and market regulation  in power grids
% A unifying energy-based approach to optimal power dispatch and market regulation

% Real-time dynamic pricing in power systems with time-varying frequencies and voltages


%%% AUTHOR AND THANKS
% \author{Michael~Shell,~\IEEEmembership{Member,~IEEE,}
%         John~Doe,~\IEEEmembership{Fellow,~OSA,}
%         and~Jane~Doe,~\IEEEmembership{Life~Fellow,~IEEE}% <-this % stops a space
% \thanks{M. Shell was with the Department
% of Electrical and Computer Engineering, Georgia Institute of Technology, Atlanta,
% GA, 30332 USA e-mail: (see http://www.michaelshell.org/contact.html).}% <-this % stops a space
% \thanks{J. Doe and J. Doe are with Anonymous University.}% <-this % stops a space
% \thanks{Manuscript received April 19, 2005; revised August 26, 2015.}}
%%%

\author{Tjerk~Stegink, Claudio~De~Persis,~\IEEEmembership{Member,~IEEE,} and~Arjan~van~der~Schaft,~\IEEEmembership{Fellow,~IEEE}% <-this % stops a space
	\thanks{This work is supported by the NWO (Netherlands Organisation for Scientific Research) programme  \emph{Uncertainty Reduction in Smart Energy Systems (URSES)}.}% <-this % stops a space
	\thanks{T.W. Stegink and C. De Persis are with the Engineering and Technology institute Groningen (ENTEG), 
		University of Groningen, 9747 AG Groningen, the Netherlands.
		{\tt\small \{t.w.stegink, c.de.persis\}@rug.nl}}% Nijenborgh 4 
	\thanks{A.J. van der Schaft is with the Johann Bernoulli Institute for Mathematics and Computer Science, University of Groningen, Nijenborgh 9, 9747 AG Groningen,
		the Netherlands.
		{\tt\small a.j.van.der.schaft@rug.nl}}%
}



%%%
% \markboth{Journal of \LaTeX\ Class Files,~Vol.~14, No.~8, August~2015}%
% {Shell \MakeLowercase{\textit{et al.}}: Bare Demo of IEEEtran.cls for IEEE Journals}
%%%

\markboth{%Transactions on automated control,~Vol.~1, No.~1, \today}
} %January~26,~2015}%
{Shell \MakeLowercase{\textit{et al.}}: Bare Demo of IEEEtran.cls for IEEE Journals}




\maketitle


\begin{abstract}
	In this paper a unifying energy-based approach \blue{is provided} to the modeling and stability analysis of power systems coupled with market dynamics. % , which is based on the \pH\ framework
We consider a standard model of the power network with a third-order model for the synchronous generators involving voltage dynamics. By applying the primal-dual gradient method  to a  social welfare \blue{optimization}, a distributed dynamic pricing algorithm is obtained, which can be naturally formulated in \pH\ form. By interconnection with the physical model a closed-loop \pH\ system is obtained, whose properties are exploited to prove asymptotic stability to the set of optimal points. This result is extended to the case that also general nodal power constraints are included into the social welfare problem. Additionally, the case of line congestion and power transmission costs in acyclic networks is covered\alert{ as well}. Finally, % several variations to the basic dynamics pricing algorithm are proposed that improve the convergence of the controller and to make it better implementable.  
a dynamic pricing algorithm is proposed that does not require knowledge about the power supply and demand. 

% NO DAPI ANYMORE
%Finally, we provide \pH\ descriptions and analysis of the well studied distributed averaging proportional integral (DAPI) and certain internal-model-based  controllers, which solve an optimal frequency regulation problem.
\end{abstract}


% \fi
\IEEEpeerreviewmaketitle

\section{Introduction}
%INTERNAL MODEL BASED AND GRADIENT METHOD 
%Two frequently used controller designs are reformulated in the \pH\ setting. Moreover, this approach proves to be more general. 

\IEEEPARstart{P}{rovisioning} energy has become increasingly complicated due to several reasons, including the increased share of renewables. As a result, the generators operate more often near their capacity limits and transmission line congestion occurs more frequently. 

%coupling physics with markets
One effective approach to alleviate some of these challenges is to use real-time dynamic pricing as a control method \cite{dynpricehis}. This feedback mechanism can be used to encourage the consumers to change their usage when in some parts of the grid (\alert{or }\emph{control areas}) it is difficult for the generators and the network to match the demand. 

Real-time dynamic pricing also allows producers and consumers to fairly share utilities and costs associated with the generation and consumption of energy among the different control areas. The challenge of achieving this in an optimal manner while the grid operates within its capacity limits, is called the \emph{social welfare problem} \cite{kiani_anna}, \cite{kiani2012hierarchicalR1}. 

Many of the existing dynamic pricing algorithms focus on the economic part of optimal supply-demand matching \cite{kiani_anna,CDC2010_Stability}. However, if market mechanisms are used to determine the optimal power % generator
  dispatch (with near real-time updates of the dispatch commands)  \alert{there will exist }dynamic coupling \blue{occurs} between the market update process and the physical response of the physical power network dynamics \cite{alv_meng_power_coupl_market}. 

Consequently, under the assumption of market-based dispatch, it is essential to consider the stability of the coupled system incorporating both market operation and electromechanical power system dynamics simultaneously. 

% \alert{For that reason, we will study the social welfare problem, with the additional requirement of achieving zero frequency deviation with respect to the nominal value (e.g. 50 Hz) in the physical power network.} % even in the presence of time-varying voltages.

% While there is vast lite

%\subsubsection{Goals}
While on this subject a vast literature is already available, the aim of this paper is to provide a rigorous and unifying passivity-based stability analysis\blue{.} \alert{of the  power network coupled with market dynamics. }We focus on a \blue{more accurate and} higher order model for \blue{the} physical power network than conventionally used in the literature. %  \blue{which provides a more accurate description % of the % physical power network
  % % dynamics
  % compared to the classical swing equations \cite{powsysdynwiley}, \cite{kundur}}.
In particular, \blue{we use a} third-order model \alert{used }for the synchronous generators \blue{including} voltage dynamics. \blue{% This will allow to consider voltage control like in e.g. \cite{de2015modular} within this same system.
  \new{As a result}, market dynamics, frequency dynamics and voltage dynamics are considered simultaneously.}

\blue{Finally}, we propose variations of the basic controller design that, among other things, allow the incorporation of capacity constraints on the generation and demand of power and on the transmission lines, and enhance the transient dynamics of the closed-loop system. The approach taken in this paper is to model both \new{the dynamic pricing controller as well as the physical network in a \pH\ way, emphasizing energy storage and power flow}. This provides a unified framework \blue{for the modeling, analysis and control of power networks with market dynamics}, with possible extensions to more refined models of the physical power network, including for example turbine dynamics.

\subsection{Literature review}
%\subsubsection{Alvarado linearization}
The coupling between a  high-order dynamic power network and market dynamics has been studied before in \cite{alv_meng_power_coupl_market}. Here a fourth-order model of the synchronous generator is used in conjunction with turbine and exciter dynamics, which is coupled \new{to} a simple model describing the market dynamics. The results established in \cite{alv_meng_power_coupl_market} are  based on an eigenvalue analysis of the linearized system.  

%\subsubsection{Third-order model}
It is shown in \cite{trip2016internal} that the third-order model (often called the \emph{flux-decay model}) for describing the power network, as used in the present paper, admits a useful passivity property that allows for a rigorous stability analysis of the interconnection with optimal power dispatch controllers, even in the presence of time-varying demand.

%\subsubsection{Gradient-method based controllers}
A common way to solve a general optimization problem like the social welfare problem is by applying the % discrete-time or continuous-time
primal-dual gradient method \cite{arrow_gradmethod,feijer-paganini}, \cite{jokic2009constrained}. % The literature on the gradient method has become quite extensive over the last decades, starting with the monograph \cite{arrow_gradmethod}.
Also in power grids this is a commonly used approach to design optimal distributed controllers, see e.g. \cite{AGC_ACC2014,mallada2014distributed,you2014reverse,zhang1achieving,zhangpapa,zhangpapaautomatica}. The problem formulations vary in these papers, with the focus being on either the generation side \cite{AGC_ACC2014,you2014reverse}, the load side \cite{mallada2014distributed}, \cite{zhao2015distributedAC}, \cite{mallada2014optimal} or both \cite{zhang1achieving,zhangpapa,zhang_papa_tree,zhangpapaautomatica}. We will elaborate on these references in the following two paragraphs. 

%Also the description of the power system is different. The models vary from linear power flows
 % \blue{In all these references the voltages in the power network are assumed to be constant and many use a second-order (non)linear model for the power network.}

% A common idea is that the market dynamics can be interpreted as a gradient method applied to a certain welfare maximization problem \cite{zhangpapaautomatica}.

%\paragraph{Power network as gradient dynamics}
A vast literature focuses on linear power system models coupled with gradient-method-based controllers \cite{AGC_ACC2014,you2014reverse}, \cite{mallada2014distributed}, \cite{mallada2014optimal}, \cite{zhang1achieving}, \cite{zhao2014design}. In these references the property that the  linear power system dynamics can be formulated as a gradient method applied to a certain optimization problem is exploited. This is commonly referred to as \emph{reverse-engineering} of the power system dynamics \cite{zhang1achieving}, \cite{AGC_ACC2014}, \cite{you2014reverse}. However, this approach falls short in dealing with models involving nonlinear power flows.

Nevertheless, \cite{zhangpapa,zhang_papa_tree,zhangpapaautomatica}, \cite{zhao2015distributedAC} show the possibility to achieve optimal power dispatch in power networks with nonlinear power flows using gradient-method-based controllers. On the other hand, the controllers proposed in \cite{zhangpapa,zhang_papa_tree,zhangpapaautomatica} have restrictions in assigning the controller parameters and in addition require that the  topology of the physical network is a tree.

% \alert{In \cite{AGC_ACC2014,you2014reverse} it is  shown how the gradient method can be applied to an optimal frequency regulation problem. However, the price-based controller requires knowledge about the demand, which was assumed to be  uncertain in these references. A method to eliminate the demand from the price dynamics is proposed by \cite{AGC_ACC2014,you2014reverse},  using a change of variables  at the cost of requiring additional knowledge about the physical dynamics. }

% Here the authors formulate the network dynamics  as a gradient method applied to a certain optimization problem.} Also in many other references the results use the property that the power system dynamics can be formulated as a gradient method, \cite{mallada2014distributed}, \cite{mallada2014optimal}, \cite{zhang1achieving}. However, this is not possible when dealing with more complex power networks with nonlinear power flows.

% \cite{zhang1achieving} deals with linear power flows and line congestion in cyclic graphs. 

\subsection{Main contributions}
%{\color{red} In our previous work we have shown that the gradient method fits within the port-Hamiltonian framework \cite{LHMNLC}.}

The contribution of this paper is to propose a novel energy-based approach to the problem that differs substantially from the aforementioned works. We proceed along the lines of \cite{LHMNLC}, \cite{stegink2015port}, where a \pH\ approach to the design of gradient-method-based controllers in power networks is proposed. In those papers it is shown that both the power network as well as the controller designs admit a port-Hamiltonian representation which are then interconnected to obtain a closed-loop \pH\ system. 

After showing that the third-order dynamical model \alert{adopted }\blue{describing} \alert{to describe }the power network admits a \pH\ representation, we provide a systematic method to design gradient-method-based controllers that \blue{is} able to balance power supply and demand while maximizing the social welfare at steady state. This design is carried out first by establishing the optimality conditions associated with the social welfare problem. Then the continuous-time gradient method is applied to obtain the \pH\ form of the \blue{dynamic pricing} controller. Then, following \cite{LHMNLC}, \cite{stegink2015port}, the market dynamics is coupled \blue{to} the physical power network in a \emph{power-preserving} manner so that all the trajectories of  the closed-loop system  converge to the desired synchronous solution and to optimal power dispatch. 

Although the proposed controllers share similarities with others presented in the literature, the way in which they are interconnected to the physical network, which is based on passivity, is to the best of our knowledge new.  Moreover, they show several advantages.

\subsubsection{Physical model}
Since our approach is based on passivity and does not require to reverse-engineer the power system dynamics as a primal-dual gradient dynamics, it allows to deal with more complex nonlinear models of the power network. More specifically, the physical model for describing the power network in this paper admits nonlinear power flows and time-varying voltages, and is more accurate and reliable than the classical second-order model \cite{powsysdynwiley}, \cite{kundur}, \cite{sauerpai1998powersystem}. 

In addition, most of the results that are established in the present paper are valid for the case of nonlinear power flows and cyclic networks, in contrast to e.g. \cite{zhang1achieving}, \cite{AGC_ACC2014}, \cite{you2014reverse}, \cite{zhao2014design}, \cite{zhao2014design}, where the power flows are linearized and e.g. \cite{zhangpapaautomatica}, \cite{zhang_papa_tree}, \cite{zhangpapa} where the physical network topology is a tree. Moreover, in the aforementioned references the voltages are assumed to be constant. While the third-order model for the power network as considered in this paper has been studied before using passivity based techniques \cite{trip2016internal}, the combination with gradient method based controllers is novel.\alert{, making this a key contribution of the present paper. } In addition, the stability analysis does not rely on linearization and is based on energy functions which allow us to establish rigorous stability results. 

\subsubsection{State transformation}
In \cite{AGC_ACC2014}, \cite{you2014reverse} it is shown how a state transformation of the closed-loop system can be used to eliminate the \blue{information about the} demand from the controller dynamics, which \blue{improves implementation of} the resulting controller.  We pursue this idea and show % Since we do not formulate the network dynamics as optimization algorithms, 
that the same kind of state transformation \alert{to eliminate the demand from the controller dynamics }can also be used for more complex physical models as considered in this paper. This \blue{avoids} the requirement of knowing the demand to determine the market price.

% \subsubsection{Unifying approach}
% However, while the optimization problems vary from case to case, the controller design procedure based on the gradient method remains similar. Providing a systematic and unifying approach to design such controllers is one of the key contributions of this paper. 

% \blue{This paper continues along the lines of our previous work \cite{LHMNLC}, \cite{stegink2015port}, where it shown that the price-based controller admit a \pH\ representation, which in particular admits a useful passivity property.}

\subsubsection{Controller parameters}
In the present paper we show that both the physical power network as well as the dynamic pricing controllers admit a \pH\ representation, and in particular are passive systems. As a result, the interconnection between the controller and the nonlinear power system is \emph{power-preserving}, implying passivity of the closed-loop system as well. Consequently, we do not have to \blue{impose any condition on} controller design parameters for guaranteeing asymptotic stability\blue{, contrary to} \cite{zhangpapa,zhang_papa_tree,zhangpapaautomatica}.

%\alert{Moreover it allows to deal with complex large-scale systems. }

\subsubsection{\PH\ framework}
Because of the use of the \pH\ framework, the proposed controller designs have the potential to deal with more complex models for the power network compared to the model described in this paper. \blue{As long as the more complex model remains \pH, the results continue to be valid.} This may lead to inclusion of, for example, turbine dynamics or automatic voltage regulators in the analysis, although this is beyond the scope of \blue{the present} paper. Furthermore, higher order models for the synchronous generator could be considered. % a study that is not explored in this paper.
\\
% This may include for example voltage regulation, although this is beyond the scope of this paper.}

\noindent In addition, we propose various extensions to the basic controller design that have not been investigated in the aforementioned references.

\subsubsection{Transmission costs}
In addition to nodal power constraints and line congestion, we also consider the possibility of including power transmission costs into the social welfare problem. Including such costs may in particular be useful in reducing energy losses or the risk of a breakdown of certain transmission lines. %This extension has, to the authors' best knowledge, not been studied before.

\subsubsection{Non-strict convex objective functions}
By relaxing the conditions on the objective function, we show that also  non-strict convex/concave cost/utility functions can be considered respectively. In addition, the proposed technique allows to add damping in the gradient method based controller which may improve the convergence rate of the closed-loop system.

\subsubsection{Barrier functions}
We highlight the possibility to use barrier functions to enforce the trajectories to stay withing the feasible region, which allows operation within the capacity constraints for all time, even during transients. This permits a more realistic application of the proposed controller design.

%We try to give a systematic straightforward approach in proving stability of power systems coupled with markets without the use of linearization around a equilibrium. The \pH\ formulation lends itself for this.  

% \subsubsection{Energy-based framework}
% Also the market dynamics admit a passivity property as shown in our previous work \cite{LHMNLC}. 

% \alert{The main advantage of the port-Hamiltonian framework is the use of passivity based techniques.}

% \subsubsection{Energy-based approach}
% With respect to the literature focusing on both optimal generation and load control \cite{zhang1achieving,zhangpapa,zhang_papa_tree,zhangpapaautomatica}, perhaps the results established in  \cite{zhangpapaautomatica} are the most similar to the results presented in this paper. In contrast to \cite{zhangpapaautomatica}, we design a gradient-method-based controller in such a way that there exists a \emph{power-preserving} interconnection between the physical system and the controller. \alert{This allows for a unifying approach to design gradient method based controllers which is one of the key contributions of this paper.} In addition, this makes the stability analysis of the closed-loop system convenient. 

%OUTLINE
%\subsubsection{Outline}
The remainder of this paper is organized as follows. In Section \ref{sec:prel} the preliminaries are stated. Then the basic dynamic pricing algorithm is discussed % a distributed gradient method is applied to a social welfare problem
in Section \ref{sec:gradmeth} and convergence of the closed-loop system is proven. Variations of the basic controller design are discussed in Section \ref{sec:vari-basic-contr} where in Section \ref{sec:nodcong} nodal power congestion is included into the social welfare problem, and in Section \ref{sec:linecong} the case line congestion for the acyclic power networks is discussed.  % The \pH\ descriptions of the DAPI controller and an internal-model-based controller is discussed in Section \ref{sec:intDAPI}, and equivalence between them is proven.
A dynamic pricing algorithm is proposed in Section \ref{sec:impl} which does not require knowledge about the power supply and demand. In Section \ref{sec:relax-strict-conv} the possibility to relax the convexity assumption and to improve the transient dynamics of the basic controller is discussed.  Finally, the conclusions and suggestions for future research are discussed in Section \ref{sec:concl-future-rese}.

\section{Preliminaries}\label{sec:prel}

\blue{\subsection{Notation}
\label{sec:notations}
Given a symmetric matrix $A\in\mathbb R^{n\times n}$, we write $A>0 \ (A\ge0)$ to indicate that $A$ is a positive (semi-)definite matrix. The  set of positive real numbers is  denoted by $\mathbb R_{>0}$ and likewise the set of  vectors in $\mathbb R^n$ whose elements are positive by $\mathbb R^n_{>0}$. We write $A\vee B$ to indicate that either condition $A$ or condition $B$ holds, e.g., $a>0\vee b>0$ means that either $a>0$ or $b>0$ holds. Given $v\in\mathbb R^n$ then  $v\succeq0$  denotes the element-wise inequality. The notation $\1\in\mathbb R^n$ is used for the vector whose elements are equal to 1. Given a twice-differentiable function $f:\mathbb R^n\to\mathbb R^n$ then the Hessian of $f$ evaluated at $x$ is denoted by $\nabla^2f(x)$. Given a vector $\eta\in\mathbb R^m$, we denote by  $\Sin\eta\in\mathbb R^m$ the element-wise sine function. Given a  differentiable function $f(x_1,x_2),x_1\in\mathbb R^{n_1},x_2\in \mathbb R^{n_2}$ then $\nabla f(x_1,x_2)$ denotes the gradient of $f$ with respect to $x_1,x_2$ evaluated at $x_1,x_2$ and likewise $\nabla_{x_1}f(x_1,x_2)$ denotes the gradient of $f$ with respect to $x_1$.} \new{%\subsection{Omega-limit sets}
Given a solution $x$ of $\dot x=f(x)$, where $f:\mathbb R^n\to\mathbb R^n$ is a Lebesgue measurable function and locally bounded,  the \emph{omega-limit set} $\Omega(x)$ is defined as \cite{cherukuri2016asymptoticcaratheodory}
\begin{align*}
    \Omega(x)&:=\Big\{\bar x\in\mathbb R^n \ | \ \exists \{t_k\}_{k=1}^\infty \subset [0,\infty) \\&\text{with } \lim_{k\to\infty}t_k=\infty \text{ and } \lim_{k\to\infty}x(t_k)=\bar x\Big\}.
  \end{align*}
}

\subsection{Power network model}
Consider a power grid consisting of $ n $ buses. The network is represented by a connected and undirected graph $ \G = (\V, \E) $, where the nodes, $ \V = \{1, . . . , n\} $, is the set of buses and the edges, $ \E = \{1,\ldots , m\} \subset \V \times \V $, is the set of transmission lines connecting the buses. The $k$-th edge connecting nodes $i$ and $j$ is denoted as $k=(i,j)=(j,i)$.  The ends of edge $ k $ are arbitrary labeled with a ‘+’ and a ‘-’, so that the incidence matrix $D$ of the \blue{resulting directed graph} is given by 
\begin{align*}
D_{ik}=\begin{cases}
+1 &\text{if $i$ is the positive end of edge $k$}\\
-1 &\text{if $i$ is the negative end of edge $k$}\\
0 & \text{otherwise.}
\end{cases}
\end{align*}
%we use here the same convention as in \cite{swing-claudio}. 
Each bus represents a control area and is assumed to
have a controllable power supply and demand. The dynamics at each bus is assumed to be given by  \cite{powsysdynwiley}, \cite{trip2016internal}
\begin{equation}
\begin{aligned}
\dot  \delta_i&= \w_i\\
  M_i \dot \w_i&=-\sum_{j\in\mathcal N_i}B_{ij}E_{qi}'E_{qj}'\sin \delta_{ij}-A_i\w_i+P_{gi}-P_{di}\\
T_{di}'\dot E_{qi}'&=E_{fi}-(1-(X_{di}-X_{di}')B_{ii})E_{qi}'\\&-(X_{di}-X_{di}')\sum_{j\in\mathcal N_i}B_{ij}E_{qj}'\cos \delta_{ij},
\end{aligned}\label{eq:3SGmultimachine}
\end{equation}
which is commonly referred to as the \emph{flux-decay model}. Here we use a similar notation as used in \new{established} literature on power systems \cite{powsysdynwiley}, \cite{kundur}, \cite{anderson1977}, \cite{sauerpai1998powersystem}.  
%Which is commonly referred to as the one-axis multi-machine model. 
See Table \ref{tab:par3SG} for a list of symbols used in the model \eqref{eq:3SGmultimachine} and throughout the paper.
\begin{table}[H]
  \centering
  \begin{tabular}[c]{ll} %\hline
 $\delta_i$ & Voltage angle  \\ 
 $\w_i^b$ & Frequency  \\ 
 $\w^n$ & Nominal frequency \\ 
 $\w_i$ & Frequency deviation  $\w_i:=\w_i^b-\w^n$\\ 
 $E_{qi}'$ & $q$-axis transient internal voltage \\ %\hline
 $E_{fi}$ & Excitation voltage \\
 $P_{di}$ & Power demand  \\ 
 $P_{gi}$ & Power generation\\%\hline
 $M_i$ & Moment of inertia \\ 
 $\mathcal N_i$ & Set of buses connected to bus $i$ \\
 $A_i$ & Asynchronous damping constant \\ 
  $B_{ij}$ & Negative of the susceptance of transmission line $(i,j)$\\  
  $B_{ii}$ & Self-susceptance \\  
 $X_{di}$ & $d$-axis synchronous reactance of generator $i$\\
 $X_{di}'$ & $d$-axis transient reactance of generator $i$\\
 $T_{di}'$ & $d$-axis open-circuit transient time constant %\\[0.5mm]\hline
\end{tabular}\vspace{1mm}
\caption{Parameters and state variables of model \eqref{eq:3SGmultimachine}.}  
\label{tab:par3SG}
\end{table}
\begin{myass}
By using the power network model \eqref{eq:3SGmultimachine} the following assumptions are made, which are standard in a broad range of literature on power network dynamics \cite{powsysdynwiley}.
\begin{itemize}
\item Lines are purely inductive, i.e., the conductance is zero. This
assumption is generally valid for the case of high
voltage lines connecting  different control areas.
%\item A balanced load condition is assumed, implying that the  three phase network can be analyzed by a single phase.
\item The grid is operating around the synchronous frequency which implies $\w_i^b\approx\w^n$ for each $i\in\V$.
% \item  Nodal voltages $ V_i $ are constant.
% \item  Reactive power flows are ignored.
% \item  A balanced load condition is assumed, implying that the
% three phase network can be analyzed by a single phase.
\blue{\item In addition, we assume for simplicity that the excitation voltage $E_{fi}$ is constant for all $i\in\V$. }
\end{itemize}
\end{myass}
Define the voltage angle differences between the buses by $\eta=D^T\delta$. 
%Define the angle difference between buses $i$ and $j$ by $\eta_k=\delta_i-\delta_j$ where edge $k$ corresponds to $(i,j)\in\mathcal E$. 
Further define the angular momenta by $ p:=M\w$, where $\w=\w^b-\blue{\1}\w^n$ are the (aggregated) frequency deviations and $M=\diag_{i\in\V}\{M_i\}$ are the moments of inertia. Let $\Gamma(E_q')=\diag_{k\in\E}\{\gamma_k\}$ and $\gamma_k=B_{ij}E_{qi}'E_{qj}'=B_{ji}E_{qi}'E_{qj}'$ where  $k$ corresponds to the edge between node $i$ and $j$. Then we can write \eqref{eq:3SGmultimachine} more compactly as \cite{trip2016internal}
\begin{equation}\label{eq:3SGcomp}
\begin{aligned}
      \dot \eta     & =D^T\w                               \\
  M   \dot \w & =-D\Gamma(E_q') \Sin \eta -A\w+P_g-P_d \\
T_{d}'\dot E_q'     & =-F(\eta)E_q'+E_f
\end{aligned}
\end{equation}
where \blue{$A=\diag_{i\in\V}\{A_i\}, P_g=\col_{i\in\V}\{P_{gi}\}, P_d=\col_{i\in\V}\{P_{di}\}, T'_d=\diag_{i\in\V}\{T_{di}'\}, E'_q=\col_{i\in\V}\{E_{qi}'\}, E_f=\col_{i\in\V}\{E_{fi}\}$}. \blue{For a given $\eta$,} the components of   \blue{the matrix $F(\eta)\in\mathbb R^{n\times n}$} are defined as 
\begin{equation}
\begin{aligned}
  F_{ii}(\eta)&=\frac{1}{X_{di}-X_{di}'}+B_{ii}, & i&\in\V\\
  F_{ij}(\eta)&=-B_{ij}\cos \eta_k=F_{ji}(\eta), & k&= (i,j)\in\E
\end{aligned}\label{eq:defFeta}
\end{equation}
and $F_{ij}\blue{(\eta)}=0$ otherwise. \alert{In addition, $\Sin(\cdot)$ denotes the element-wise sine function. }Since for realistic power networks $X_{di}>X_{di}'$, and $B_{ii}=\sum_{j\in\mathcal N_i}B_{ij}>0$ for all $i\in \V$, it follows that $F(\eta)>0$ for all $\eta$ \cite{powsysdynwiley}, \cite{kundur}.
% \begin{myrem}
%   {Note that here that $B_{ii}=\sum_{j\in\mathcal N_i}B_{ij}$ and $B_{ij}\leq 0$ compared to} \cite{trip2016internal}. 
% \end{myrem}

Considering the  physical energy\footnote{For aesthetic reasons we define the Hamiltonian $H_p$ as $\w^n$ times the physical energy as the factor $1/\w^n$ appears in each of the energy functions. As a result, $H_p$ has the dimension of power instead of energy.}
%In fact, since the factor $1/\w^n$ appears in each of the energy functions, it is convenient to define $H_p$ as  $\w^n$ times the physical energy.} 
stored in the generator and the transmission lines respectively, we define the Hamiltonian as
\begin{equation}
\begin{aligned} H_p&=\frac1{2}\sum_{i\in\V}\left(M_i^{-1}p_i^2+\frac{(E_{qi}'-E_{fi})^2}{X_{di}-X_{di}'}\right)\\
&+\frac1{2}\sum_{\mathclap{ k= (i,j)\in \E}}B_{ij}\left((E_{qi}')^2+(E_{qj}')^2-2E_{qi}'E_{qj}'\cos \eta_k \right)
\end{aligned}\label{eq:H3SG}
\end{equation}
where $\eta_{k}=\delta_i-\delta_j$. The first term of the Hamiltonian $H_p$ represents the \blue{shifted} kinetic energy stored in the rotors of the generators and the second term corresponds to the magnetic energy stored in the generator circuits. Finally, the last term of $H_p$ corresponds to the magnetic energy stored in the inductive transmission lines.

By \eqref{eq:H3SG}, the system \eqref{eq:3SGcomp} can be written in port-Hamiltonian form \blue{\cite{phsurvey}} as
\begin{equation}\label{eq:3SGmultiph}
\begin{aligned}
\dot x_p&=
  \begin{bmatrix}
    0 & D^T & 0                                & \\
    -D    & -A  & 0                                & \\
    0    & 0  & -R_q
  \end{bmatrix}
\nabla H_p
+  \begin{bmatrix}
    0&0\\
    I&-I\\
    0&0
  \end{bmatrix}u_p\\
y_p&=\begin{bmatrix}
    0&I &0\\
    0&-I&0
  \end{bmatrix}\nabla H_p=
          \begin{bmatrix}
            \w\\
            -\w
          \end{bmatrix}
\end{aligned}
\end{equation}
where $x_p=\col(\eta,p,E_q')$, $u_p=\col(P_g,P_d)$ and 
\begin{align*}
       R_q   & =(T_{d}')^{-1}(X_{d}-X_{d}')>0, \\
       \blue{T_{d}'} & =\diag_{i\in\V}\{\blue{T_{di}'}\}>0,    \\
X_{d}-X_{d}' & =\diag_{i\in\V}\{X_{di}-X_{di}'\}>0.
\end{align*}
% In the sequel we write the dynamics of the physical system \eqref{eq:3SGmultiph} more compactly as 
% \begin{align*}
%   \dot x_p&=(J_p-R_p)\nabla H_p+g_p u_p\\
%   y_p&= g_p^T\nabla H_p,
% \end{align*}
% where $J_p=-J_p^T, R_p=R_p^T\geq0$.
For a study on the stability and equilibria of the flux-decay model \eqref{eq:3SGmultiph}, based on the Hamiltonian function \eqref{eq:H3SG}, we refer to \cite{trip2016internal}. The stability results established in \cite{trip2016internal} rely on the following assumption.
%Given a constant input $u=\bar u$, an additional assumption is required for proving local asymptotic stability of \eqref{eq:phswingeq}  \cite{swing-claudio}.
\begin{myass} %[security constraint]
	\label{ass:sec1}
	Given a constant input $u_p=\bar u_p$. There exists an equilibrium $(\bar \eta,\bar p,\bar E_q')$ of \eqref{eq:3SGmultiph} that satisfies $\bar \eta\in\text{im}\, D^T$, $\bar \eta\in (- \pi/2, \pi/2)^m$ and $\nabla^2H(\bar \eta,\bar p,\bar E_q')>0$.
\end{myass}
The assumption $\bar \eta\in (- \pi/2, \pi/2)^m$ is standard in studies on power grid stability and is also referred to as a security constraint \cite{trip2016internal}. In addition, the Hessian condition guarantees the existence of a local storage function around the equilibrium.  The  following result, which
establishes decentralized conditions for checking the positive
definiteness of the Hessian, was proven \blue{in} \cite{de2015modular}:
\begin{myprop}\label{prop:Hessiancond}
Let $\bar E_{qi}'\in \mathbb R_{>0}^n$ and $\bar \eta\in(-\pi/2,\pi/2)^m$. If for all $i\in\V$ we have
\begin{align*}
&\frac{1}{X_{di}-X_{di}'}+B_{ii}+\sum_{\mathclap{k=(i,j)\in\E}}B_{ij}\frac{E_{qj}'\sin^2\bar\eta_k}{E_{qi}'\cos\bar \eta_k}\\&>
\sum_{\mathclap{k=(i,j)\in \E}}B_{ij}\cos\bar \eta_k \left(1+\frac{\bar E_{qi}}{\bar E_{qj}}\tan^2\bar \eta_k\right)>0
\end{align*}
then $\nabla^2 H_p(\bar x_p)>0$. 
\end{myprop}
\noindent It can be verified that the condition stated in Proposition \ref{prop:Hessiancond} is satisfied if the following holds \cite{de2015modular}:
\begin{itemize}
\item the generator reactances are small compared to the transmission line reactances
\item the voltage (angle) differences are small.
\end{itemize}
Remarkably, these conditions hold for a typical operation point in power transmission networks. 

%In standard operation conditions in large-scale power networks this is generally the case (REFERENCE). 

% SCALED FREQUENCY
%Note that in the modeling of the swing equations the nominal frequency in included in the moments of intertia and damping coefficients. In fact we have that $M=\w^n\hat M, A=\w^n \hat A$ where $\hat M, \hat A$ are the real physical moments of inertia and damping coefficients. As a consequence, $M^{-1}p=\hat M^{-1}\frac{p}{\w^n}=\frac{\w}{\w^n}$. 

\subsection{Social welfare problem}
\label{sec:soci-welf-probl}
We define the social welfare by $S(P_g,P_d):=U(P_d)-C(P_g)$, which consists of a  utility function $U(P_d)$ of the power consumption $P_d$ and the  cost $C(P_g)$ associated to the power production $P_g$.  We assume that $C(P_g),U(P_d)$ are strictly convex and strictly concave functions respectively. %  It should be stressed that the utility and cost functions  considered here are more general compared the most of the literature. 
\begin{myrem}
  It is also possible to include mutual costs and utilities among the different control areas, provided that the convexity/concavity assumption is satisfied.
\end{myrem}
The objective is to maximize the social welfare while achieving zero frequency deviation.  
%The latter condition is necessary to ensure zero frequency deviation of \eqref{eq:phswingeq} at steady state. We consider the following convex minimization problem: 
By analyzing the equilibria of \eqref{eq:3SGmultimachine}, it follows that a necessary condition for zero frequency deviation is $\1^T P_d=\1^T P_g$ \cite{trip2016internal}, i.e.,  the total supply must match the total demand. It can be noted that $(P_g,P_d)$ is a solution to the latter equation if and only if there exists a $v\in \mathbb R^{m_c}$ satisfying $D_\comm v-P_g+P_d=0$ where $D_\comm\in\mathbb R^{n\times m_c}$ is the incidence matrix of some connected \emph{communication} graph with $m_c$ edges and $n$ nodes.  Because of the latter equivalence, we  consider the following convex minimization problem: 
\begin{subequations}\label{eq:minprobbasic}
\begin{align}
\min_{P_g,P_d,v} \ \  & -S(P_g,P_d)=C(P_g)-U(P_d)\\
\text{s.t.}\ \  & D_\comm v-P_g+P_d=0.
\end{align}
\end{subequations}

%Although we consider voltage dynamics in the power system model, for simplicity we do not include constraints on the voltage or reactive in the above optimization problem. } % this could be done in future research indeed as discussed. While we do not discuss voltage regulation here, in this paper we pave the in doing so by proving Lyapunov functions for the higher-order model case. %

\blue{\begin{myrem}
  %We emphasize that the main focus of this paper is on optimal active power dispatch, resulting in optimization problem \eqref{eq:minprobbasic}. While omitting the details, we stress that it is also possible to consider similar objectives for the reactive power, see for example \cite{de2015modular}.%, although this is beyond the scope of this paper. 

Although this paper focuses on optimal \emph{active} power sharing, we stress that it is also possible to consider (optimal) \emph{reactive} power sharing simultaneously, see e.g. \cite{de2015modular} \new{for more details}.
%objectives like optimal reactive power sharing as done in \cite{de2015modular}, although this is beyond the scope of this paper. %the focus of this paper is maximizing the social welfare 
\end{myrem}}

The \blue{Lagrangian corresponding to \eqref{eq:minprobbasic}} is given by 
\begin{align}\label{eq:L}
\mathcal L &= C(P_g)-U(P_d)+\lambda^T(D_\comm v-P_g+P_d)
\end{align}
with Lagrange multipliers $\lambda \in\mathbb R^n$. The resulting first-order optimality conditions are given by the Karush–Kuhn–Tucker (KKT) conditions 
\begin{equation}\label{eq:KKTcondbasic}
\begin{aligned}
	\nabla C(\bar P_g)-\bar \lambda  & =0, \\
	-\nabla U(\bar P_d)+\bar \lambda & =0, \\
	D_\comm^T\bar \lambda                  & =0, \\
	D_\comm \bar v-\bar P_g+\bar P_d       & =0.
\end{aligned}
\end{equation}
Since the minimization problem is convex, strong duality holds and it follows that $(\bar P_g,\bar P_d,\bar v)$ is an optimal solution to \eqref{eq:minprobbasic} if and only if there exists an $\bar \lambda\in\mathbb R^n$ that satisfies \eqref{eq:KKTcondbasic} \cite{convexopt}. 

% \subsection{Consumers dynamics}
% \label{sec:prosumer-dynamics}
% Suppose that  the power consumers want to maximize their utility based on the electricity price $\mathsf p$. The associated optimization problem is  given by 
% \begin{equation}\label{eq:consmin}
%   \min_{P_d}\ -U(P_d)+\mathsf{p}^TP_d.
% \end{equation}
% However, the consumers adapt their behavior over a certain time-period such their dynamics can be described by 
% \begin{equation}
% \begin{aligned}
%   \tau_d\dot P_d&=\nabla U(P_d)-\mathsf{p}\\
% \end{aligned}\label{eq:consumersdyn}
% \end{equation}
% where $\tau_d=\diag_{i\in\V}\{\taP_{di}\}>0$ are the time-scales associated with the consumers response to price variations. The  dynamics of the power consumption \eqref{eq:consumersdyn} is a generalized version as described by  \cite{dynpricehis}, \cite{kiani_anna}, \cite{alv_meng_power_coupl_market}, and others. Observe that for a constant price $\mathsf p$ the trajectories of \eqref{eq:consumersdyn} converge to the optimal point of  \eqref{eq:consmin}, see also \eqref{eq:KKTcondbasic}. 

% The objectives are now to design a dynamic pricing scheme for $\mathsf p$ and a controller for the power production $P_g$ such that the overall social welfare is maximized and the frequency is regulated. 

\section{Basic primal-dual gradient controller}\label{sec:gradmeth}
In this section we design the basic dynamic pricing algorithm which will be used as the starting point for the controllers designs discussed in Section \ref{sec:vari-basic-contr}.  Its dynamics is obtained by applying the primal-dual gradient method \cite{arrow_gradmethod,AGC_ACC2014,zhangpapa} to the minimization problem \eqref{eq:minprobbasic}, \blue{resulting in}
\begin{subequations}\label{eq:graddistalvswing}
\begin{align}
\tau_{g}\dot P_g&=-\nabla C(P_g)+\lambda+u_{c}^g\\
\tau_d\dot P_d&=\nabla U(P_d)-\lambda+u_{c}^d\\
\tau_{ v}\dot  v&=-D_\comm^T\lambda \label{eq:vdyn}\\
\tau_\lambda\dot \lambda&=D_\comm  v-P_g+P_d.\label{lamdadyn}
\end{align}
\end{subequations}
Here we introduce additional inputs $u_c=\col (u_{c}^g,u_{c}^d)$ which are  to be specified later on, and $\tau_c:=\text{blockdiag}(\tau_g,\tau_d,\tau_v,\tau_\lambda)>0$ are controller design parameters.  Recall from Section \ref{sec:soci-welf-probl} that there is freedom in choosing a communication network and the associated incidence matrix. Depending on the application, one may prefer  all-to-all communication where the underlying graph is complete, or communication networks where its associated graph is a star, line or cycle graph. In addition, $\tau_c$ determines the converge rate of the dynamics \eqref{eq:graddistalvswing}; a large $\tau_c$ gives a slow convergence rate whereas a small $\tau_c$ gives a fast convergence rate. 

% Suppose that $u_c=0$.
% Then
Observe\alert{ furthermore} that the dynamics \eqref{eq:graddistalvswing} has a clear economic interpretation \cite{dynpricehis},  \cite{kiani_anna}, \cite{alv_meng_power_coupl_market}: each power producer \blue{aims at maximizing} their own profit, which occurs whenever their individual marginal cost is equal to the local price $\lambda_i+u_{ci}^g$. At the same time, each consumer maximizes its own utility but is penalized by the local price $\lambda_i-u_{ci}^d$. 

The equations \eqref{eq:vdyn}, \eqref{lamdadyn} represent the distributed dynamic pricing mechanism where the quantity $v$ represents a \emph{virtual} power flow along the edges of the communication graph with incidence matrix $D_c$. We emphasize \emph{virtual}, since $v$ may not correspond to the real physical power flow as the communication graph may be different than the physical network topology.   Equation \eqref{lamdadyn} shows that the local price $\lambda_i$ rises if the  power demand plus power outflow at node $i\in\V$ is greater than the local power supply plus power inflow of power at node $i$ and vice versa. The inputs $u_{c}^g, u_{c}^d$ are interpreted as additional penalties or prices that are assigned to the power producers and consumers respectively. These inputs can be chosen appropriately to compensate for the frequency deviation in the physical power network as we will show now.

%Note that \eqref{eq:graddistalvswing} is a distributed controller where $\lambda_i$ acts as a price in control area $i\in\V$ and $v$ represents the information exchange of the differences of the prices $\lambda$ along the edges of the communication graph. 

To this end, define the variables $x_c=(x_g,x_d,x_v,x_\lambda) = (\tau_gP_g,\tau_dP_d,\tau_ v v,\tau_\lambda\lambda)=\tau_cz_c$ and note that, in the sequel, we interchangeably write the system dynamics in terms of both $x_c$ and $z_c$ for ease of notation. In these new variables %  and note that, in the sequel, we interchangeably write the system dynamics in terms of energy variables (denoted by $x$) and co-energy variables (denoted by $ z $) for ease of notation. 
% A main starting point is to note that
the dynamics \eqref{eq:graddistalvswing} admits a  natural \pH\  representation  \cite{LHMNLC}, which is given by  
\begin{align}\label{eq:dynpricsys}
\dot x_c&=
\begin{bmatrix}
	0  & 0 & 0       & I          \\
	0  & 0 & 0       & -I         \\
	0  & 0 & 0       & -D_\comm^T \\
	-I & I & D_\comm & 0
\end{bmatrix}\nabla H_c(x_c)+\nabla S(z_c) \nonumber\\
&+\begin{bmatrix}
I&0 \\
0& I\\
0&0\\
0&0
\end{bmatrix}u_c\\
y_c&=\begin{bmatrix}
I&0&0&0\\
0&I&0&0
\end{bmatrix}\nabla H_c(x_c)=\begin{bmatrix}
P_g\\P_d
\end{bmatrix},\nonumber\\
H_c(x_c)&=\frac{1}{2}x_c^T\tau_c^{-1}x_c. \label{eq:hc}
\end{align}
Note that the system \eqref{eq:dynpricsys} is indeed a \pH\ system\footnote{Strictly speaking, \eqref{eq:dynpricsys} is an \emph{incremental} \pH\ system \cite{phsurvey}.} since  $S$ is concave  and therefore satisfies the \blue{incremental passivity} property 
\begin{align*}
(z_1-z_2)^T(\nabla S(z_1)-\nabla S(z_2))\leq0, \ \forall z_1,z_2\in\mathbb R^{3n+m_c}.
\end{align*}
% The system \eqref{eq:dynpricsys} is written compactly as 
% \begin{align*}
%   \dot x_c&=J_c\nabla H_c(x_c)+\nabla S(z_c)+g_cu_c\\
%   y_c&=g_c^T\nabla H_c(x_c)
% \end{align*}
% where $J_c=-J_c^T$.
The \pH\ controller \eqref{eq:dynpricsys} is interconnected to the physical network \eqref{eq:3SGmultiph} by taking % We obtain a power-preserving interconnection between \eqref{eq:3SGmultiph} and \eqref{eq:dynpricsys} by choosing
$u_c=-y_p, u_p=y_c$. Define the extended vectors of % (co-)energy
variables by 
\begin{align}\label{xz}
x&:=
\begin{bmatrix}
	I & 0 & 0 &  0      & 0      & 0      & 0            \\
	0 & M & 0 &  0      & 0      & 0      & 0            \\
	0 & 0 & I &  0      & 0      & 0      & 0            \\
	0 & 0 & 0 &  \tau_g & 0      & 0      & 0            \\
	0 & 0 & 0 &  0      & \tau_d & 0      & 0            \\
	0 & 0 & 0 &  0      & 0      & \tau_v & 0            \\
	0 & 0 & 0 &  0      & 0      & 0      & \tau_\lambda
\end{bmatrix}
\begin{bmatrix}
\eta\\
\w\\
E_q'\\
P_g\\
P_d\\
v\\
\lambda
\end{bmatrix}=:\tau z.
\end{align} 
Then the closed-loop \pH\ system takes the form 
\begin{equation}\label{eq:clphsys2}
\begin{aligned}
\dot x&=
\begin{bmatrix}
0  & D^T & 0    & 0  & 0  & 0       & 0          \\
-D & -A  & 0    & I  & -I & 0       & 0          \\
0  & 0   & -R_q & 0  & 0  & 0       & 0          \\
0  & -I  & 0    & 0  & 0  & 0       & I          \\
0  & I   & 0    & 0  & 0  & 0       & -I         \\
0  & 0   & 0    & 0  & 0  & 0       & -D_\comm^T \\
0  & 0   & 0    & -I & I  & D_\comm & 0
\end{bmatrix}\nabla H(x)\\&+\nabla S(z),
\end{aligned}
\end{equation}
% \begin{equation}\label{eq:clphsys2}
% \begin{aligned}
% \dot x&=
% \begin{bmatrix}
% J_p-R_p&g_pg_c^T\\
% -g_cg_p^T&J_c
% \end{bmatrix}\nabla H(x)+\nabla S(z)
% \end{aligned}
% \end{equation}
where $H=H_p+H_c$ is equal to  the sum of the energy function \eqref{eq:H3SG} corresponding to the physical model, and the controller Hamiltonian \eqref{eq:hc}.
%In the next section we prove the asymptotic stability to the optimal point of the latter closed-loop system.
In the sequel we write \eqref{eq:clphsys2} more compactly as 
\begin{align*}%\label{eq:clphsys2comp}
  \dot x=(J-R)\nabla H(x)+\nabla S(z),
\end{align*}
where $R=R^T\geq0, J=-J^T$. 
We define the equilibrium set of \eqref{eq:clphsys2}, expressed in the  variable $z$, by 
\begin{align}\label{eq:omega2}
\eqset_1&=\{ \bar z \ | \  \bar z \text{ is an equilibrium of } \eqref{eq:clphsys2} \}.
\end{align}
Note that each $\bar z\in\eqset_1$ satisfies the  optimality conditions \eqref{eq:KKTcondbasic} and simultaneously the zero frequency constraints of the physical network \eqref{eq:3SGmultiph} given by $\bar \w=0$. Hence, $ \eqset_1$ corresponds to the desired equilibria, and the next theorem states  the convergence to this set of optimal points.

% Note that $\eqset_1$ is equal to the set of all $\bar z$ that satisfy the  optimality conditions \eqref{eq:KKTcondbasic} and simultaneously the zero frequency constraints of the physical network \eqref{eq:3SGmultiph} given by $\bar \w=0$ and
% \begin{align*}
% -D\Gamma \sin \bar \eta+\bar P_g-\bar P_d=0.
% \end{align*}
% Hence, $ \eqset_1$ corresponds to the desired equilibria, and the next theorem states  the convergence to this set of optimal points.

%For proving asymptotic {stability of the closed-loop }system \eqref{eq:clphsys1} an additional assumption is required.
%\begin{ass} %[security constraint]
%\label{ass:sec2}
%$\eqset_1\neq \emptyset$ and there exists a $(\bar \eta,\bar p,\bar \lambda)\in \eqset_1$ such that $\bar \eta_k\in (- \pi/2, \pi/2)$ for all $k\in\E$.
%\end{ass}
%The first part of the assumption $(\eqset_1\neq \emptyset)$ is to ensure that there exists an equilibrium of the closed-loop system. The last part of the assumption $(\bar \eta_k\in(-\pi/2,\pi/2))$ is standard in studies on power grid stability and is also referred to as a security constraint \cite{swing-claudio-voltage}.

%, with $\eqset_2$ defined in \eqref{eq:omega2},

%\subsection{Stability}
\begin{mythm}\label{thm:grad}  For every $\bar z\in  \eqset_1$ \new{satisfying Assumption \ref{ass:sec1}} there exists a \alert{open }neighborhood $\new{\Upsilon}$ around $\bar z$ where all trajectories $z$ satisfying \eqref{eq:clphsys2} with initial conditions in $\new{\Upsilon}$ converge to the set $\eqset_1$. \new{In addition, the convergence of each such trajectory is to a point.} 
\end{mythm}

\begin{proof} Let $\bar z\in\eqset_1$ and define the shifted Hamiltonian $\bar H$  around $\bar x:=\tau\bar z$ as \cite{phsurvey}, \cite{trip2016internal}
\begin{align}\label{eq:shHam}
&\bar H(x)=H(x)-(x-\bar x)^T\nabla H(\bar x)-H(\bar x).
\end{align}
After rewriting, the closed-loop \pH\ system \eqref{eq:clphsys2} is equivalently described by 
\begin{equation*}
\begin{aligned}
\dot x&=
% \begin{bmatrix}
% J_p-R_p&g_pg_c^T\\
% -g_cg_p^T&J_c
% \end{bmatrix}
(J-R)\nabla \bar H(x)+\nabla S(z)-\nabla S(\bar z).
\end{aligned}
\end{equation*}
The shifted Hamiltonian $\bar H$ satisfies 
\begin{equation}
\begin{aligned}\label{eq:shdissineq}
\dot{\bar H}&=-\w^T A\w+(z-\bar z)^T(\nabla S(z)-\nabla S(\bar z))\\
&-(\nabla_{E_q'}\bar H)^TR_q\nabla_{E_q'}\bar H\leq0,
\end{aligned}
\end{equation}
where equality holds if and only if $\nabla_{E_q'}\bar H(x)=\nabla_{E_q'} H(x)=0,\w=0, P_g=\bar P_g, P_d=\bar P_d$ since $S(z)$ is strictly concave in $ P_g $ and $P_d$.  \alert{Bearing in mind Assumption \ref{ass:sec1}, it is observed that  $\bar H(\bar x)=0$ and $\bar H(x)>0$ for all $x\neq \bar x$ in a  sufficiently small \alert{open }neighborhood  around $\bar x$. }\new{Bearing in mind Assumption \ref{ass:sec1}, it is observed that  $ \nabla^2 H(x)=\nabla^2 \bar H(x)>0$ for all $x$ in a  sufficiently small open neighborhood around $\bar x$.}
Hence, as $ \dot {\bar H} \leq  0 $, there exists a compact \new{sub}level set $\Upsilon$ of $\bar H$ around $\bar z$ \new{contained in such neighborhood}, which is forward invariant. \alert{Take a neighborhood  of $\bar z$ contained in such a level set.} By LaSalle’s \new{invariance} principle, \new{each} the solution \new{with initial conditions in $\new{\Upsilon}$} converges to the largest invariant set $\mathcal S$ contained in
$ \Upsilon \cap \{z\ | \ \nabla_{E_q'} H(x)=0,\w = 0,\ P_g = \bar P_g , \ P_d =\bar P_d\}. $ On such invariant set $\lambda=\bar \lambda $ and $\eta,v,E_q'$ are constant. \alert{We conclude that }\new{Hence,} $z$ converges to $ \mathcal S\subset \eqset_1$ as $t\to \infty$. 

\new{Finally, we prove that the  convergence of each solution  of \eqref{eq:clphsys2} initializing in $\new{\Upsilon}$ is to a point. This is equivalent to proving that its omega-limit set $\Omega(x)$ is a singleton. Since the solution $x$ is bounded, $\Omega(x)\neq \emptyset$ by the Bolzano-Weierstrass theorem \cite{rudin1964principles}. By contradiction, suppose now that there exist two distinct point in $\Omega(x)$, say $\bar x_1,\bar x_2\in\Omega(x),\bar x_1\neq \bar x_2$.  Then there exists $\bar H_1(x),\bar H_2(x)$ defined by \eqref{eq:shHam} with respect to $\bar x_1,\bar x_2$ respectively and scalars $c_1,c_2\in \mathbb R_{>0}$ such that $\bar H_{1}^{-1}(\leq c_1):=\{x \ | \ \bar H_1(x)\leq c_1\}, \bar H_{2}^{-1}(\leq c_2):=\{x \ | \ \bar H_2(x)\leq c_2\}$ are disjoint and compact as the Hessian of $\bar H_1,\bar H_2$ is positive definite in  the neighborhood $\Upsilon$. Since  each trajectory $z$ converges to $\mathcal Z_1$ as proven above, it follows that  $\tau^{-1}\bar x_1,\tau^{-1}\bar x_2\in \eqset_1$. Together with $\bar x_1\in\Omega(x)$, this implies that there exists a finite time  $t_1>0$ such that $x(t)\in \bar H_{1}^{-1}(\leq c_1) $ for all $t\geq t_1$ as the set  $\bar H_{1}^{-1}(\leq c_1)$ is invariant by the dissipation inequality \eqref{eq:shdissineq}. Similarly, there exists a finite time $t_2>0$ such that $x(t)\in \bar H_{2}^{-1}(\leq c_2) $ for all $t\geq t_2$. This implies that the solution $x(t)$ satisfies $x(t)\in \bar H_{1}^{-1}(\leq c_1)\cap \bar H_{1}^{-1}(\leq c_1)=\emptyset$ for $t\geq \max(t_1,t_2)$ which is a contradiction.  This concludes the proof. 
% there exists a time $t_N$ such that $x(t_N)$ This contradicts the fact that $\bar x_1, \bar x_2$  are both elements of $\Omega(x)$.
}
%$\mathcal {QWERTYUIOPASDFGHJKLZXCVBNM}$
\end{proof}

\section{Variations in the controller design}
\label{sec:vari-basic-contr}
In this section we propose several variations and extensions of the controller designed in the previous section. These include, among other things, the possibility to incorporate nodal power constraints, and line congestion in conjunction with transmission costs into the social welfare problem. 

\subsection{Including nodal power constraints}
\label{sec:nodcong}
The results of Section \ref{sec:gradmeth} can be extended to the case where nodal constraints on the power production and consumption are included into the optimization problem \eqref{eq:minprobbasic}. To this end, consider the social welfare problem 
\begin{subequations}\label{eq:minprobnodcong2}
\begin{align}
\min_{P_g,P_d,v} \ \  & -S(P_g,P_d):=C(P_g)-U(P_d)\\
\text{s.t.}\ \  & D_\comm v-P_g+P_d=0,\\
 &g(P_g,P_d) \preceq 0\label{eq:gineq}
\end{align}
\end{subequations}
where $g:\mathbb R^{2n}\to \mathbb{R}^l$ is a convex function\blue{.} \alert{ and $\preceq$ denotes the element-wise inequality.} 
\begin{myrem}\label{rem:g1g2g3g4}
Note that \eqref{eq:gineq}  captures the convex  inequality constraints \blue{considered} in the existing literature. For example, by choosing $g$ as 
\begin{align*}
g(P_g,P_d)=\begin{bmatrix}
g_1(P_g,P_d)\\
g_2(P_g,P_d)\\
g_3(P_g,P_d)\\
g_4(P_g,P_d)
\end{bmatrix}=\begin{bmatrix}
P_g-P_g^{\max}\\
P_g^{\min}-P_g\\
P_d-P_d^{\max}\\
P_d^{\min}-P_d\\
\end{bmatrix},
\end{align*}
the resulting inequality constraints \eqref{eq:gineq} become  $P_g^{\min}\preceq P_g\preceq P_g^{\max}, \ P_d^{\min}\preceq P_d\preceq P_d^{\max}$ which, among others, are used in \cite{zhang1achieving,zhangpapa,zhang_papa_tree}.
\end{myrem}
In the sequel, we assume that \eqref{eq:minprobnodcong2} satisfies Slater's condition \cite{convexopt}. As a result, $(\bar P_g,\bar P_d,\bar v)$ is an optimal solution to \eqref{eq:minprobnodcong2} if and only if there exists $\bar \lambda\in \mathbb R^n, \bar \mu\in \mathbb R^l_{\geq0}$ satisfying the following KKT optimality conditions: 
\begin{equation}\label{eq:KKTcondnodal2}
\begin{aligned}
	\nabla C(\bar P_g)-\bar \lambda +\frac{\p g}{\p P_g}(\bar P_g,\bar P_d)\bar \mu                                         =0,&          \\
	-\nabla U(\bar P_d)+\bar \lambda +\frac{\p g}{\p P_d}(\bar P_g,\bar P_d)\bar \mu             =0,&          \\
	D_\comm \bar v-\bar P_g+\bar P_d                 =0, \qquad           	D_\comm^T\bar \lambda        =0,&          \\
g(\bar P_g,\bar P_d) \preceq 0, \qquad \bar \mu\succeq 0, \qquad 
	\bar \mu^Tg(\bar P_g,\bar P_d)               =0.&
\end{aligned}
\end{equation}
%Here we use the convention
%\begin{align*}
%\frac{\p f}{\p x}(x)=
%\begin{bmatrix}
%\frac{\p f_1}{\p x_1}&\ldots&\frac{\p f_m}{\p x_1}\\
%\vdots & \ddots & \vdots\\
%\frac{\p f_1}{\p x_n} & \ldots & \frac{\p f_m}{\p x_n}
%\end{bmatrix}
%\end{align*}
%where $f:\mathbb{R}^n\to \mathbb R^m$. 
Next, we introduce the following subsystems \cite{LHMNLC}, \cite{jokic2009constrained}
\begin{equation}\label{eq:xipassys2}
\begin{aligned}
  \dot x_{\mu_i}&=(g_i(w_i))_{\mu_i}^+:= \begin{cases}
 g_i(w_i) & \text{if } \mu_i>0\\
 \max \{0,g_i(w)\} & \text{if } \mu_i=0
 \end{cases}\\
 y_{\mu_i}&=\nabla g_i(w_i)\nabla H_{\mu_i}(x_{\mu_i}), \  H_{\mu_i}(x_{\mu_i})=\frac{1}{2}x_{\mu_i}^T\tau_{\mu_i}^{-1}x_{\mu_i}
\end{aligned}
\end{equation} 
with  state $x_{\mu_i}:=\tau_{\mu_i}\mu_i\in\mathbb R_{\geq0}$, outputs $y_{\mu_i}\in\mathbb R^l$,  inputs $w_i\in\mathbb R^{2n}$, and $i\in\mathcal I:=\{1,\ldots,l\}$. \blue{Here $g_i(.)$ is the $i$'th entry of the vector-valued function $g(.)=\col_{i\in\V}\{g_i(.)\}$.} Note that, for a given $i\in\mathcal I$ and for a constant input $\bar w_i$, the equilibrium set $ \eqset_{\mu_i} $ of \eqref{eq:xipassys2} is characterized by all $(\bar \mu_i,\bar w_i)$ satisfying
\begin{align}\label{eq:omegamu}
g_i(\bar w_i)\leq 0, \qquad \bar \mu_i\geq 0, \qquad \bar \mu_i=0\vee g_i(\bar w_i)=0.
\end{align}
More formally, for $i\in\mathcal I$ the equilibrium set $\eqset_{\mu_i}$ of \eqref{eq:xipassys2} is given by 
\begin{align*}
\eqset_{\mu_i}:=\{(\bar \mu_i,\bar w_i)\  | \ (\bar \mu_i,\bar w_i) \text{ satisfies } \eqref{eq:omegamu}\}.
\end{align*}
\begin{myrem}
In case the inequality constraints of Remark \ref{rem:g1g2g3g4} (e.g. $P_g\preceq P_g^{\max}$)  are considered, the subsystems \eqref{eq:xipassys2} take the decentralized form 
\begin{equation}
\begin{aligned}
  \dot x_{\mu_i}&=(P_{gi}-P_{gi}^{\max})_{\mu_i}^+= \begin{cases}
 P_{gi}-P_{gi}^{\max} & \text{if } \mu_i>0\\
 \max \{0,P_{gi}-P_{gi}^{\max}\} & \text{if } \mu_i=0
 \end{cases}\\
 y_{\mu_i}&=\nabla H_{\mu_i}(x_{\mu_i}), \quad   H_{\mu_i}(x_{\mu_i})=\frac{1}{2}x_{\mu_i}^T\tau_{\mu_i}^{-1}x_{\mu_i}, \quad i\in \V,
\end{aligned}
\end{equation} 
and similar expressions can be given for the remaining inequalities $P_g^{\min}\preceq P_g, \ P_d^{\min}\preceq P_d\preceq P_d^{\max}$.
\end{myrem}
The subsystems \eqref{eq:xipassys2} have the following passivity property \cite{LHMNLC}.
\begin{myprop}[\cite{LHMNLC}]\label{prop:musys}
Let $i\in \mathcal I,(\bar \mu_i,\bar w_i)\in\eqset_{\mu_i}$  and define $  \bar y_{\mu_i}:=\nabla g_i(\bar w_i)\bar \mu_i. $ Then \eqref{eq:xipassys2} is passive with respect to the shifted external port-variables $\tilde w_i:=w_i-\bar w_i, \tilde y_{\mu_i}:=y_{\mu_i}-\bar y_{\mu_i}$. \blue{Additionally, $( \mu_i, w_i)\to\eqset_{\mu_i}$ as $t\to\infty$  for $(\mu_i, w_i), w_i=\bar w_i$ satisfying \eqref{eq:xipassys2}.}
\end{myprop}

Consider again system \eqref{eq:clphsys2}
\begin{equation}\label{eq:clphsysnodalio}
\begin{aligned}
\dot x&=
(J-R)\nabla H(x)+\nabla S(z)+G^Tu\\
y&=G\nabla H(x)=\begin{bmatrix}
P_g\\P_d
\end{bmatrix}, \quad G=\begin{bmatrix}
0&0&I&0&0&0\\
0&0&0&I&0&0
\end{bmatrix}
% \dot x&=
% \begin{bmatrix}
% 0  & D^T & 0    & 0  & 0  & 0       & 0          \\
% -D & -A  & 0    & I  & -I & 0       & 0          \\
% 0  & 0   & -R_q & 0  & 0  & 0       & 0          \\
% 0  & -I  & 0    & 0  & 0  & 0       & I          \\
% 0  & I   & 0    & 0  & 0  & 0       & -I         \\
% 0  & 0   & 0    & 0  & 0  & 0       & -D_\comm^T \\
% 0  & 0   & 0    & -I & I  & D_\comm & 0
% \end{bmatrix}\nabla H(x)\\
% &+\nabla S(z)+B^Tu\\
% y&=B\nabla H(x)=\begin{bmatrix}
% P_g\\P_d
% \end{bmatrix}, \quad B=\begin{bmatrix}
% 0&0&I&0&0&0\\
% 0&0&0&I&0&0
% \end{bmatrix}
\end{aligned}
\end{equation}
with the state $x$ defined by \eqref{xz}, where we introduce an  additional input $u\in\mathbb R^{2n}$ and output $y\in\mathbb R^{2n}$.
\begin{myrem}\label{rem:passive}
Note that for any steady state $(\bar x,\bar u)$ of \eqref{eq:clphsysnodalio}, the  latter system is passive with respect to the shifted external port-variables $\tilde u:=u-\bar u, \ \tilde y=y-\bar y, \ \bar y:=G\nabla H(\bar x)$, using the storage function 
\begin{align}\label{eq:barH}
\bar H(x):=H(x)-(x-\bar x)^T\nabla H(\bar x)-H(\bar x).
\end{align}
\end{myrem}
We interconnect the subsystems \eqref{eq:xipassys2} to \eqref{eq:clphsysnodalio} in a power-preserving way by 
\begin{align*}
w_i=w=y \quad  \forall i\in\mathcal I, \qquad u=-\sum_{i\in\mathcal I}y_{\mu_i}
\end{align*}
to obtain the closed-loop system
\begin{subequations}\label{eq:nodalclsys}
	\begin{align}
	\dot \eta&=D^T\w\\
  M   \dot \w & =-D\Gamma(E_q') \Sin \eta -A\w+P_g-P_d \\
T_d'\dot E_q'     & =-F(\eta)E_q'+E_f\\
	\tau_g\dot P_g&=-\nabla C(P_g)+\lambda- \frac{\p g}{\p P_g}( P_g, P_d) \mu             -\w\\
	\tau_d\dot P_d&=\nabla U(P_d)-\lambda- \frac{\p g}{\p P_d}( P_g, P_d) \mu+\w\\
	\tau_v\dot v&=-\nabla C_T(v)- D^T\lambda\\
	\tau_\lambda\dot{\lambda}&=D v-P_g+P_d\\
	\tau_{\mu_i}\dot \mu_i&=(g_i(P_g,P_d))_{\mu_i}^+, \qquad i\in\mathcal I.
	\end{align}
\end{subequations}
Observe that the equilibrium set $\eqset_2$ of \eqref{eq:nodalclsys}, if expressed in the co-energy variables, is characterized by all $(\bar z,\bar \mu)$ that satisfy \eqref{eq:KKTcondnodal2} in addition to $\bar \w=0$, $-D\Gamma(\bar E_q') \Sin \bar \eta+\bar P_g-\bar P_d=0,-F(\bar \eta)\bar E_q'+E_f=0$, and therefore corresponds to the desired operation points.

Since both the subsystems \eqref{eq:xipassys2} and the system \eqref{eq:clphsys2} admit an \emph{incrementally passivity} property with respect to their steady states, the closed-loop system inherits the same property provided that an equilibrium of \eqref{eq:nodalclsys} exists.
\begin{mythm} \label{thm:ineqcontgrad} For every $(\bar z,\bar \mu)\in  \eqset_2$ \new{satisfying Assumption \ref{ass:sec1}} there exists a \alert{open }neighborhood $\new{\Upsilon}$ of $(\bar z,\bar \mu)$ where all trajectories $z$ satisfying \eqref{eq:nodalclsys} with initial conditions in $\new{\Upsilon}$ converge to the set $\eqset_2$ \new{and the convergence of each such trajectory is to a point.}
\end{mythm}
\begin{proof} Let $ (\bar z,\bar \mu)\in\eqset_2 $ and consider the shifted Hamiltonian $\bar H_e$ around $ (\bar x,\bar x_\mu)=(\tau \bar z,\tau_\mu \bar \mu) $ defined by 
\begin{align*}
\bar H_e(x,x_\mu)&:=\bar H(x)+\sum_{i\in\mathcal I}\bar H_{\mu_i}(x_{\mu_i})=\bar H(x)+\frac{1}{2}\tilde x_\mu^T\tau_{x_\mu}^{-1}\tilde x_\mu
\end{align*}
where $\tilde x_\mu:=x_\mu-\bar x_\mu$ and $\bar H$ is defined by \eqref{eq:barH}. By Proposition \ref{prop:musys} and Remark \ref{rem:passive}, the time-derivative of $\bar H_e$ satisfies
\begin{align*}
\dot {\bar H}_e\leq \tilde u^T\tilde y+\tilde w^T\sum_{i\in\mathcal I}\tilde y_{\mu_i}=\tilde u^T\tilde y-\tilde u^T\tilde y=0
\end{align*}
where equality holds only if $P_g=\bar P_g, P_d=\bar P_d, \w =0, \nabla_{E_q'} H(x)=0$. On the largest invariant set where $\dot {\bar {H}}_e=0$ it follows by the second statement of Proposition \ref{prop:musys} that $\mu=\bar \mu$. As a result, $\lambda=\bar \lambda$ and $v,\eta,E_q'$ are constant on this invariant set. Since the right-hand side of \eqref{eq:xipassys2} is discontinuous and takes the same form as in \cite{cherukuri2016asymptoticcaratheodory}, we can apply the invariance principle for discontinuous Caratheodory systems \blue{\cite[Proposition 2.1]{cherukuri2016asymptoticcaratheodory}} 
% By applying a Caratheodory\footnote{We are aware of \cite{cherukuri2016asymptoticcaratheodory} and the fact that the right-hand side of \eqref{eq:xipassys2} is discontinuous.} variant of LaSalle's invariance principle  \cite{cherukuri2016asymptoticcaratheodory}, we 
to conclude that $(z,\mu)\to \eqset_2$ as $t\to\infty$. \new{By following the same line of arguments as in the proof of Theorem \ref{thm:grad}, convergence of each trajectory to a point is proven.}

%By applying a Caratheodory variant of LaSalle's invariance principle we conclude that $(x,\lambda,{\mu}) \to \eqset_1$ as $t\to\infty$ (\cite{cherukuri2016asymptoticcaratheodory}).
\end{proof}
\blue{\begin{myrem}
Theorem \ref{thm:ineqcontgrad} uses the Caratheodory variant of the Invariance Principle which requires that the Caratheodory solution % \footnote{\blue{In fact, Caratheodory solutions of \eqref{eq:nodalclsys} are  considered.}}
of \eqref{eq:nodalclsys} is unique and that its omega-limit set is invariant \cite{cherukuri2016asymptoticcaratheodory}. % Although details are omitted, it can be proven that these conditions % required
%in \cite[Proposition 2.1]{cherukuri2016asymptoticcaratheodory} 
\new{These requirements} are indeed satisfied % when considering \eqref{eq:nodalclsys}
 % which can be verified
 by extending Lemmas 4.1-4.4 of \cite{cherukuri2016asymptoticcaratheodory} to the case where equality constraints and \emph{nonstrict} convex/concave (utility) functions are considered in the optimization problem \cite[equation (3)]{cherukuri2016asymptoticcaratheodory}, \new{noting that these lemmas only require convexity/concavity  instead of their strict versions}. In particular, by adding a quadratic \new{function of the} Lagrange multipliers associated with  the equality constraints %\footnote{Consisting of the quadratic terms of the Lagrange multipliers corresponding with the equality constraints.} 
to the Lyapunov function, it can be proven that monotonicity of the primal-dual dynamics with respect to primal-dual optimizers as stated in \cite[Lemma 4.1]{cherukuri2016asymptoticcaratheodory}  holds for this more general case as well, see also \cite{LHMNLC}, \cite{cherukuri2015saddle}. % the dissipation equality appearing in the proof of

% the primal-dual dynamics obtained from the equality constraints do contribute to the energy balance of \cite[Lemma 4.1]{cherukuri2016asymptoticcaratheodory} once a quadratic term associated with the Lagrange multipliers of the equality constraints is added to the Lyapunov function, see also \cite[Section 3.2]{LHMNLC}. 

  % In \cite{cherukuri2016asymptoticcaratheodory} the (utility) functions appearing in the optimization problem are \emph{strictly} convex/concave and there are no equality constraints considered. Nevertheless, while omitting the details, the results of \cite[Theorem 4.5]{cherukuri2016asymptoticcaratheodory} can be extended to the more general case by extending Lemmas 4.1, 4.2, 4.3 and 4.4 of \cite{cherukuri2016asymptoticcaratheodory}. In particular, the convergence is to the largest invariant set (which in the more general case may not correspond with the set op primal-dual optimizers). Hence, the invariance principle for Caratheodory systems \cite[Proposition 2.1]{cherukuri2016asymptoticcaratheodory} can safely be applied in the proof of Theorem \ref{thm:ineqcontgrad}. 
\end{myrem}}
\begin{myrem}
  Instead of using the hybrid dynamics \eqref{eq:xipassys2} for dealing with the inequality constraints \eqref{eq:gineq}, we can instead introduce the so called \emph{barrier functions} $B_i=-\blue{\nu}\log(-g_i(P_g,P_d))$ that are added to the objective function \cite{convexopt}. Simultaneously,  the corresponding inequalities in the social welfare problem \eqref{eq:minprobnodcong2} are removed to obtain the modified convex optimization problem 
  \begin{equation}\label{eq:intpoint}
\begin{aligned}
\min_{P_g,P_d,v} \ \  & -S(P_g,P_d)-\nu\sum_{i\in\V}\log(-g_i(P_g,P_d))\\
\text{s.t.}\ \  & D_\comm v-P_g+P_d=0. %\label{eq:gineqbarrier}
\end{aligned}
\end{equation}
 Here $\nu>0$ is called the \emph{barrier parameter} and is usually chosen small. By applying the primal-dual gradient method to \eqref{eq:intpoint} it can be shown that, if the system is initialized in the interior of the feasible region, i.e. where \eqref{eq:gineq} holds, then the trajectories of the resulting gradient dynamics remain within the feasible region and the system converges to a suboptimal value of the social welfare \cite{arrow_gradmethod}, \cite{convexopt}, \cite{wang2011control}. However, if Slater's condition holds,  this suboptimal value which depends on $\blue{\nu}$ converges to the optimal value of the social welfare problem as $\blue{\nu}\to 0$  \cite{convexopt}. The particular advantage of using barrier functions is to avoid the use of an hybrid controller and to enforce that the trajectories remain within the feasible region  for all future time. 
\end{myrem}

\subsection{Including line congestion and transmission costs} \label{sec:linecong}
The previous section shows how to include nodal power constraints into the social welfare problem. In case the network is acyclic, line congestion and power transmission costs can be incorporated into the optimization problem as well. 

%To this end, we introduce constraints on the (non-physical) power flow $v$.  

To this end, define the (modified) social welfare by $U(P_d)-C(P_g)-C_T(v)$ where the convex function $C_T(v)$ corresponds to the power transmission cost. If security constraints on the transmission lines are included as well, the optimization problem \eqref{eq:minprobbasic} modifies to 
\begin{subequations}\label{eq:minprob}
\begin{align}
\min_{P_g,P_d,v} \ \  & -S(P_g,P_d,v):=C(P_g)+C_T(v)-U(P_d)\\
\text{s.t.}\ \  & D v-P_g+P_d=0\\
 &-\kappa \preceq v\preceq \kappa, \label{eq:addinconst}
\end{align}
\end{subequations}
where $\kappa \in\mathbb R^m$ satisfies the element-wise inequality $\kappa \succ 0$.  Note that in this case the communication graph is chosen to be identical with the topology of the physical network, i.e., $D_c=D$. As a result,  the additional constraints \eqref{eq:addinconst} bound the (virtual) power flow along the transmission lines as $|v_k|\leq \kappa_k, k\in\E$.  The corresponding Lagrangian is given by 
\begin{align*}
\mathcal L &= C(P_g)+ C_T(v)-U(P_d)+\lambda^T(D v-P_g+P_d)\\
&+\mu_+^T(v-\kappa)+\mu_-^T(-\kappa-v)
\end{align*}
with Lagrange multipliers $\lambda \in\mathbb R^n, \ \mu_+,\ \mu_+\in\mathbb R_{\geq0}^{m}$. The resulting KKT optimality conditions   are given by  
\begin{equation}\label{eq:KKTcond2}
\begin{aligned}
	\nabla C(\bar P_g)-\bar \lambda            =0,  \qquad   
	-\nabla U(\bar P_d)+\bar \lambda                                                                 & =0, \\
	\nabla C_T(\bar v)+D^T\bar \lambda+\bar \mu_+-\bar \mu_-                                                                                                    & =0, \\
	-\kappa \preceq \bar v                        \preceq \kappa, \qquad 	D \bar v-\bar P_g+\bar P_d                                                                  & =0, \\
	\bar \mu_+,\bar \mu_-                      \succeq 0, \quad      	\bar \mu_+^T(\bar v-\kappa)                =0,  \quad         
	\bar \mu_-^T(-\kappa-\bar v) & =0.
\end{aligned}
\end{equation}
Suppose that Slater's condition holds. Then, since the optimization problem \eqref{eq:minprob} is convex, it  follows that $(\bar P_g,\bar P_d,\bar v)$ is an optimal solution to \eqref{eq:minprob} if and only if there exists $\bar \lambda \in\mathbb R^n, \ \bar \mu=\col(\bar \mu_+,\ \bar \mu_-)\in\mathbb R_{\geq0}^{2m}$ satisfying  \eqref{eq:KKTcond2} \cite{convexopt}.  

By applying the gradient method to \eqref{eq:minprob} in a similar manner as before and 
%the controller dynamics amounts to
%\begin{align*}
%\tau_g\dot P_g&=-\nabla C(P_g)+\lambda-\w\\
%\tau_d\dot P_d&=\nabla U(P_d)-\lambda+\w\\
%\tau_v\dot v&=-\nabla C_T(v)- D^T\lambda-\mu_++\mu_-\\
%\tau_\lambda\dot{\lambda}&=D v-P_g+P_d\\
%\tau_+\dot \mu_+&=(v-\kappa)^+_{\mu_+}\\
%\tau_-\dot \mu_-&=(-\kappa-v)^+_{\mu_-}
%\end{align*}
%where $\mu_+,\mu_-:\mathbb R\to \mathbb R^m_{\geq 0}$ and 
%\begin{align*}
%(h(y))^+_\mu=\text{col}\begin{cases}
%h_i(y)& \text{if } \mu_i>0\\
%\max(0,h_i(y))& \text{if } \mu_i=0
%\end{cases}, \quad i=1,\ldots,q
%\end{align*}
%where $h(y):\mathbb R^m\to \mathbb R^q , \mu\in \mathbb R^q$. 
connecting the resulting controller with the physical model \eqref{eq:3SGcomp},  we obtain the following closed-loop system:
%\begin{equation}
\begin{subequations}\label{eq:secconstraint}
	\begin{align}
\dot \eta&=D^T\w\label{eq:partph1}\\
  M   \dot \w & =-D\Gamma(E_q') \Sin \eta -A\w+P_g-P_d \\
T_d'\dot E_q'     & =-F(\eta)E_q'+E_f\\
\tau_g\dot P_g&=-\nabla C(P_g)+\lambda-\w\\
\tau_d\dot P_d&=\nabla U(P_d)-\lambda+\w\\
\tau_v\dot v&=-\nabla C_T(v)- D^T\lambda-\mu_++\mu_-\\
\tau_\lambda\dot{\lambda}&=D v-P_g+P_d\label{eq:partphl}\\
\tau_+\dot \mu_+&=(v-\kappa)^+_{\mu_+}\label{eq:mup}\\
\tau_-\dot \mu_-&=(-\kappa-v)^+_{\mu_-}.\label{eq:mum}
\end{align}
\end{subequations}
The latter system  can partially be  put into a \pH\ form, since equations \eqref{eq:partph1}-\eqref{eq:partphl} can be rewritten as
\begin{equation}\label{eq:clphsys}
\begin{aligned}
\dot x&=
% {\begin{bmatrix}
% 0  & D^T & 0    & 0  & 0  & 0       & 0          \\
% -D & -A  & 0    & I  & -I & 0       & 0          \\
% 0  & 0   & -R_q & 0  & 0  & 0       & 0          \\
% 0  & -I  & 0    & 0  & 0  & 0       & I          \\
% 0  & I   & 0    & 0  & 0  & 0       & -I         \\
% 0  & 0   & 0    & 0  & 0  & 0       & -D_\comm^T \\
% 0  & 0   & 0    & -I & I  & D_\comm & 0
% \end{bmatrix}}\nabla H(x)\\&+\nabla S(z)
% +\begin{bmatrix}
% 0\\0\\0\\0\\-\mu_++\mu_-\\0
% \end{bmatrix}
(J-R)\nabla H(x)+\nabla S(z)+N\mu\\
N&=
\begin{bmatrix}
  0&0&0&0&-I&0\\
  0&0&0&0&I&0
\end{bmatrix}^T,
\end{aligned}
\end{equation}
%where $x=(\eta,M\w,\tau_gP_g,\tau_dP_d,\tau_vv,\tau_\lambda \lambda  )=:\tau z,\tau_+\mu_+$ 
where the variables $x,z$ and the Hamiltonian $H$ are respectively defined by \eqref{xz} and \eqref{eq:clphsys2} as before, and $\mu=\col(\mu_+,\mu_-)$.
%We now investigate the equilibria of \eqref{eq:secconstraint}. 
% {\begin{mylem}[\cite{trip2016internal}] \label{lem:equi}
% 	Suppose the network is acyclic. Then there exists an isolated
% %	\footnote{\color{cyan}Where $\bar \eta\in (-\pi/2,\pi/2)$.} 
% equilibrium $(\bar z,\bar \mu)$ of \eqref{eq:secconstraint} satisfying $\bar \eta_k\in (-\pi/2,\pi/2), \forall k\in\E$ which is  characterized by $(\bar z,\bar \mu)$ satisfying \eqref{eq:KKTcond2} in addition to $\bar \w=0$, and
% 	\begin{align*}
% 	-D\Gamma(\bar E_q') \Sin \bar \eta+\bar P_g-\bar P_d&=0\\
%           -F(\bar \eta)\bar E_q'+E_f&=0.
% 	\end{align*}
% \end{mylem}}
% As a result of Lemma \ref{lem:equi} which assumes that

Since the network topology is a tree (i.e. $\ker D=\{0\}$),  the equilibrium of \eqref{eq:secconstraint} satisfies $\bar v=\Gamma(\bar E_q')\Sin \bar \eta$. Hence, the controller variable $v$ corresponds to the \emph{physical} power flow of the network if the closed-loop system is at steady state. Consequently, the constraints and costs on $v$ correspond to constraints and costs of the physical power flow if the system converges to an equilibrium. 

% \alert{\begin{myrem}
%     Under Assumption \ref{ass:sec1}  it can be shown that, in case the network is acyclic, the equilibria of \eqref{eq:secconstraint}  are \emph{isolated}   \cite{trip2016internal}. However,  the equilibria are not necessarily  unique. 
% \end{myrem}
% }

%Stability analysis can be executed to obtain the following result: 
\begin{mythm}\label{thm:stab} Let the network  topology be acyclic and let $(\bar z,\bar \mu)$ be an (isolated) equilibrium of \eqref{eq:secconstraint} \new{satisfying Assumption \ref{ass:sec1}}. Then all trajectories $(z,\mu)$ of \eqref{eq:secconstraint}  initialized in a sufficiently small \alert{open }neighborhood around $(\bar z,\bar \mu)$ converge asymptotically to $(\bar z,\bar \mu)$
\end{mythm}
\begin{proof} 
Let $(\bar z,\bar \mu)$ be the equilibrium of \eqref{eq:secconstraint}. By defining the shifted Hamiltonian $\bar  H(x)$ around $ \bar x:=\tau\bar z $ by 
\begin{align*}
\bar H(x)=H(x)-( x-\bar x)^T\nabla H(\bar x)- H(\bar x)
\end{align*}
one can rewrite \eqref{eq:clphsys} as
\begin{equation}\label{eq:clphsysshifted}
\begin{aligned}
\dot x&=% \begin{bmatrix}
% 0  & D^T & 0    & 0  & 0  & 0       & 0          \\
% -D & -A  & 0    & I  & -I & 0       & 0          \\
% 0  & 0   & -R_q & 0  & 0  & 0       & 0          \\
% 0  & -I  & 0    & 0  & 0  & 0       & I          \\
% 0  & I   & 0    & 0  & 0  & 0       & -I         \\
% 0  & 0   & 0    & 0  & 0  & 0       & -D_\comm^T \\
% 0  & 0   & 0    & -I & I  & D_\comm & 0
% \end{bmatrix}\nabla \bar H(x)\\&+\nabla S(z)-\nabla S(\bar z)
% +\begin{bmatrix}
% 0\\0\\0\\0\\-\tilde \mu_++\tilde \mu_-\\0
% \end{bmatrix}
(J-R)\nabla \bar H(x)+\nabla S(z)-\nabla S(\bar z)+N\tilde \mu
\end{aligned}
\end{equation}
where $\tilde \mu:=\mu-\bar \mu$.	Consider candidate Lyapunov function 
\begin{align*}
V(x,\mu)&=\bar H(x)+\frac{1}{2}\tilde \mu_+\tau_{\mu_+}\tilde \mu_++\frac{1}{2}\tilde \mu_-^T\tau_{\mu_-}\tilde\mu_-
\end{align*}
and   observe that 
\begin{align*}
\tilde \mu_+^T(v-\kappa)^+_{\mu_+}&\leq \tilde \mu_+^T(v-\kappa)= \tilde \mu_+^T(\bar v-\kappa+\tilde v)\leq \tilde \mu_+^T\tilde v\\
\tilde \mu_-^T(-\kappa-v)^+_{\mu_-}&\leq \tilde \mu_-^T(-\kappa-v)\\
&= \tilde \mu_-^T(-\kappa-\bar v-\tilde v)\leq -\tilde \mu_-^T\tilde v.
\end{align*}
Bearing in mind \eqref{eq:clphsysshifted}, the time-derivative of $V$  amounts to
\begin{align*}
&\dot V=-\w^T A\w-(\nabla_{E_q'}H(x))^TR_q\nabla_{E_q'}H(x)\\
&+(z-\bar z)^T(\nabla S(z)-\nabla S(\bar z))\\
&-\tilde v^T\tilde \mu_++\tilde v^T\tilde \mu_- +\tilde \mu_+^T(v-\kappa)^+_{\mu_+} +\tilde \mu_-^T(-\kappa-v)^+_{\mu_-}\\
&\leq-\w^T A\w+(z-\bar z)^T(\nabla S(z)-\nabla S(\bar z))\\
&-(\nabla_{E_q'}H(x))^TR_q\nabla_{E_q'}H(x)\leq 0
\end{align*}
where equality holds only if $\nabla_{E_q'}H(x)=0,\w=0,P_g=\bar P_g,P_d=\bar P_d$. On the largest invariant set $\mathcal S$ where $\nabla_{E_q'}H(x)=0,\w=0,P_g=\bar P_g,P_d=\bar P_d$ it follows that, since the graph contains no cycles $\lambda=\bar \lambda,v=\bar v, \mu=\bar \mu$ and that $\eta,E_q'$ are constant, which corresponds to an equilibrium.  In particular $\nabla V(x,\mu)=0$ for all $(z,\mu)\in \mathcal S$ and $(\bar z,\bar \mu)\in\mathcal S$.
%  since the graph contains no cycles, $\eta=\bar \eta,\lambda=\bar \lambda, v=\bar v, \mu=\bar \mu $. 
Since by Assumption \ref{ass:sec1} we have $\nabla^2 V(\bar x,\bar \mu)>0$, it follows that $(\bar z,\bar \mu)$ is isolated. By the invariance principle for discontinuous Caratheodory systems \cite{cherukuri2016asymptoticcaratheodory} all trajectories $(z,\mu)$ of \eqref{eq:secconstraint} initializing in a sufficienly small neighborhood around $(\bar z,\bar \mu)$ satisfy $\mu\to \bar \mu$, $z\to \bar z$ as $t\to\infty$.

%\alert{It is guaranteed that $\nabla V(\bar x,\bar \mu)=0$}.
\end{proof}

%\begin{myrem}
%The above results are easily extendable such that additional inequality constraints are  incorporated as well. 
%%and is left due to space limitations. 
%For example bounds on the nodal power generation  and demand (e.g. $ P_g^{\min}\leq P_g\leq P_g^{\max}, P_d^{\min}\leq P_d\leq P_d^{\max}$) can easily be included since the stability analysis follows the same lines as stated in the proof of  Theorem \ref{thm:stab}. 
%%NODAL CONSTRAINTS FOR CYCLIC NETWORKS?
%\end{myrem}
%%{\color{cyan}

\begin{myrem}
  It \alert{should be noted that it }is possible to include nodal power constraints, line congestion and transmission costs simultaneously. However, as the results in this section are only valid for acyclic graphs, it should also be assumed for the more general case that the physical \blue{network} is a tree.  
\end{myrem}

\subsection{State transformation}
\label{sec:impl}
Consider again the minimization problem \eqref{eq:minprobbasic}. As shown before, by applying the gradient method to the social welfare problem, the closed-loop system \eqref{eq:clphsys2} is obtained. 

Note that % the controller \eqref{clgradint} cannot be implemented since
in the $\lambda$-dynamics the demand $P_d$ appears, which in practice is often uncertain. A possibility to eliminate the demand from the controller dynamics is by a state transformation \cite{AGC_ACC2014,you2014reverse}. To this end, define the new variables 
\begin{align*}
\hat x:=\begin{bmatrix}
\eta\\ p\\E_q'\\  x_g\\x_d\\x_v\\x_\theta
\end{bmatrix}=
\begin{bmatrix}
I&0&0&0&0&0&0\\
0&I&0&0&0&0&0\\
0&0&I&0&0&0&0\\
0&0&0&I&0&0&0\\
0&0&0&0&I&0&0\\
0&0&0&0&0&I&0\\
0&I&0&0&0&0&I
\end{bmatrix}x=\hat \tau \begin{bmatrix}
\eta\\ p\\ E_q'\\ P_g\\P_d\\v\\\theta
\end{bmatrix}=:\hat \tau \hat z,
\end{align*}
i.e., $x_\theta:=\tau_\theta \theta=p+x_\lambda$. Then the \pH\  system \eqref{eq:clphsys2} transforms to 
\begin{equation}\label{eq:clphsysalt}
\begin{aligned}
\dot {\hat x}&=\begin{bmatrix}
	0  & D^T & 0 & 0 & 0  & 0       & D^T        \\
	-D & -A  & 0 & I & -I & 0       & -A         \\
	0  & 0   & -R_q & 0 & 0  & 0       & 0          \\
	0  & -I  & 0 & 0 & 0  & 0       & 0          \\
	0  & I   & 0 & 0 & 0  & 0       & 0          \\
	0  & 0   & 0 & 0 & 0  & 0       & -D_\comm^T \\
	-D & -A  & 0 & 0 & 0  & D_\comm & -A
\end{bmatrix}\nabla \hat H(\hat x)\\
&+\nabla S(\hat z)
\end{aligned}
\end{equation}
with Hamiltonian
\begin{align*}
\hat H(\hat x)&=H_p+\frac{1}{2}x_g^T\tau_g^{-1}x_g+\frac{1}{2}x_d^T\tau_d^{-1}x_d\\
&+\frac{1}{2}x_v^T\tau_v^{-1}x_v+\frac{1}{2}(x_\theta-p)\tau_\lambda^{-1}(x_\theta-p).
\end{align*}
By writing the system of differential equations \eqref{eq:clphsysalt} explicitly we obtain
\begin{equation}\label{eq:secconstraintalt}
\begin{aligned}
\dot \eta&=D^T\w\\
  M   \dot \w & =-D\Gamma(E_q') \Sin \eta -A\w+P_g-P_d \\
T_d'\dot E_q'     & =-F(\eta)E_q'+E_f\\
\tau_g\dot P_g&=-\nabla C(P_g)+\tau_\lambda^{-1}(\tau_\theta \theta-M\w)-\w\\
\tau_d\dot P_d&=\nabla U(P_d)-\tau_\lambda^{-1}(\tau_\theta \theta-M\w)+\w\\
\tau_v\dot v&=- D_\comm^T\tau_\lambda^{-1}(\tau_\theta \theta-M\w)\\
\tau_\theta\dot{\theta}&=D_\comm v-D\Gamma \Sin \eta-A\w.
\end{aligned}
\end{equation}
Define $\eqset_4$ as  the set of all $\hat z^*:=(\bar \eta,\bar \w, \bar P_g,\bar P_d, \bar v, \bar \theta) $ that are an equilibrium of \eqref{eq:secconstraintalt}.
Using the previous established tools we can prove asymptotic stability to the set of optimal points $ \eqset_4 $. 
\begin{mythm}\label{thm:2} For every $\hat z^*\in \eqset_4$ \new{satisfying Assumption \ref{ass:sec1}} there exists a \alert{open }neighborhood $\new{\Upsilon}$ around $\hat z^*$ where all trajectories $\hat z$ satisfying \eqref{eq:secconstraintalt}  (or equivalently \eqref{eq:clphsysalt}) and initializing in $\new{\Upsilon}$ converge to  $\eqset_4$. \new{In addition, the convergence of each such trajectory is to a point.}
\end{mythm}
\begin{proof}
  We proceed along the same lines as \blue{in} the proof of Theorem \ref{thm:grad}. Since the stability result of Theorem \ref{thm:grad} is preserved after a state transformation, the proof is concluded.
\end{proof}
Note that the latter result holds for all $\tau_g,\tau_d,\tau_v,\tau_\lambda,\tau_\theta>0$. The controller appearing in \eqref{eq:secconstraintalt} can be simplified by choosing $\tau_\lambda=\tau_\theta=M$. As a result, the controller dynamics is described by
\begin{subequations}
\label{eq:secconstraintalt2}
    \begin{align}
      \tau_g\dot P_g&=-\nabla C(P_g)+\theta-2\w\\
      \tau_d\dot P_d&=\nabla U(P_d)-\theta+2\w\\
      \tau_v\dot v&=- D_\comm^T(\theta-\w) \label{eq:vprice}\\
      M\dot{\theta}&=D_\comm v-D\Gamma(E_q') \Sin \eta-A\w. \label{eq:tprice}
% \tau_u\dot u&=-\nabla C(u)+\theta-2\w\\
% \tau_v\dot v&=- D_\comm^T(\theta-\w)\\
% M\dot{\theta}&=D_\comm v-D\Gamma \Sin \eta-A\w.
    \end{align}
\end{subequations}
The main advantage of controller design \eqref{eq:secconstraintalt2} is that no information about the power supply and demand is required in the dynamic pricing algorithm \eqref{eq:vprice},~\eqref{eq:tprice}, where we observe that the quantity $\theta-2\w$ acts here as the electricity price for the producers and consumers. Another benefit of the proposed dynamic pricing algorithm is that, \alert{in }contrary to \cite{zhao2015distributedAC}, no information is required \blue{about} $\dot \w$. 

On the other hand,  knowledge about the physical power flows and  the power system parameters $M, A$ is required. % , B_{ij},(i,j)\in\E$ Bij is already included in the power flows
  % and the voltage (angles) $\eta,E_q'$
Determining the radius of uncertainty of these parameters under which asymptotic stability is preserved remains an open question \cite{AGC_ACC2014}\blue{; see \cite{mallada2014optimal} for results} in a similar setting  where only the damping term $A$ is assumed to be uncertain.

\subsection{Relaxing the strict convexity assumption}
\label{sec:relax-strict-conv}
By making a minor modification to the social welfare problem \eqref{eq:minprobbasic}, it is possible to relax the condition that the functions $C,U$ are \emph{strictly} convex and concave respectively. To this end, consider the optimization problem 
\begin{subequations}\label{eq:minprobbasicaug}
\begin{align}
\min_{P_g,P_d,v} \ \  & C(P_g)-U(P_d)+\frac12\rho||D_\comm v-P_g+P_d||^2\label{eq:modS}\\
\text{s.t.}\ \  & D_\comm v-P_g+P_d=0,
\end{align}
\end{subequations}
where $\rho>0$, $C(P_g)$ is convex and $U(P_d)$ is concave, which makes the optimization problem  \eqref{eq:minprobbasicaug} convex. Suppose that there exists a feasible solution to the minimization problem, then the set of optimal points of \eqref{eq:minprobbasicaug} is identical with the set of optimal points of \eqref{eq:minprobbasic} which is characterized by set of points satisfying the KKT conditions \eqref{eq:KKTcondbasic}. The corresponding augmented Lagrangian of \eqref{eq:minprobbasicaug} is given by 
\begin{align*} 
\mathcal L_p&=C(P_g)-U(P_d)-\lambda^T(D_cv+P_g-P_d)\\
&+\frac12\rho||D_cv+P_g-P_d||^2.
\end{align*}
Consequenctly, the distributed dynamics of the primal-dual gradient method applied to \eqref{eq:minprobbasicaug} amounts to 
\begin{equation}
\begin{aligned}
  \tau_g\dot P_g&=-\nabla C(P_g)+\lambda-\rho(D_cv+P_g-P_d)\\
  \tau_d\dot P_d&=\nabla U(P_d)-\lambda+\rho(D_cv+P_g-P_d)\\
  \tau_v\dot v&= D_c^T\lambda-\rho D_c^T(D_cv+P_g-P_d)\\
  \tau_\lambda\dot\lambda&=-D_cv-P_g+P_d,
\end{aligned}\label{eq:strconv}
\end{equation}
which can be written in the same port-Hamiltonian form as  \eqref{eq:clphsys2} where in this case 
\begin{align}\label{eq:Smod}
  S(P_g,P_d,v)=U(P_d)-C(P_g)-\frac12\rho||D_\comm v-P_g+P_d||^2.
\end{align}
% \alert{Note that the dynamics \eqref{eq:strconv} is still distributed as only local information is required.} 
This leads to the following result.
\begin{mythm}
  Consider the system \eqref{eq:clphsys2} where $S$ is given by \eqref{eq:Smod} and suppose that $C,U$ are convex and concave functions respectively. Then for every $\bar z\in\eqset_1$ \new{satisfying Assumption \ref{ass:sec1}}, where $\eqset_1$ is defined by \eqref{eq:omega2}, there exists a \alert{open }neighborhood $\new{\Upsilon}$ around $\bar z$ \new{wherein each trajectory} $z$ satisfying \eqref{eq:clphsys2} \new{converges to a point in} % the set
   $\eqset_1$. %\new{Moreover, the convergence of each trajectory is to a point.} 
\end{mythm}
\begin{proof}
Let $\bar z\in\eqset_1$. By the proof of Theorem \ref{thm:grad} it follows that 
  \begin{align*}
    \dot{\bar H}&=-\w^T A\w+(z-\bar z)^T(\nabla S(z)-\nabla S(\bar z))\\
&-(\nabla_{E_q'}\bar H)^TR_q\nabla_{E_q'}\bar H, 
  \end{align*}
where the second term can be written as 
\begin{align}
% & (z-\bar z)^T(\nabla S(z)-\nabla S(\bar z))\\
&\tilde P_d^T(\nabla U(P_d)-\nabla U(\bar P_d))-\tilde P_g^T(\nabla C(P_g)-\nabla C(\bar P_g)) \nonumber\\
&-\rho
  \begin{bmatrix}
    \tilde P_g\\
    \tilde P_d\\
    \tilde v
  \end{bmatrix}^T
\begin{bmatrix}
              -I     & I     & -D_c    \\
              I      & -I    & D_c     \\
              -D_c^T & D_c^T & -D_c^TD_c
            \end{bmatrix}
  \begin{bmatrix}
    \tilde P_g\\
    \tilde P_d\\
    \tilde v
  \end{bmatrix}\leq 0\label{eq:adddamp}
\end{align}
where $\tilde P_g=P_g-\bar P_g, \tilde P_d=P_d-\bar P_d, \tilde v=v-\bar v $. Hence, we obtain that  $\dot {\bar H}\leq 0$ where equality holds  only if $\w=0, \nabla_{E_q'}\bar H(x)=0 $ and  $D_c\tilde v+\tilde P_g-\tilde P_d=D_c v+P_g- P_d=0$. On the largest invariant set $\mathcal S$ where $\dot{ \bar H}_c=0$ we have $\w=0$ and $\eta,E_q'$ are constant and $(P_g,P_d,v,\lambda)$ satisfy the KKT optimality conditions \eqref{eq:KKTcondbasic}.  Therefore $\mathcal S\subset\eqset_1$ and by LaSalle's invariance principle there exists a neighborhood $\new{\Upsilon}$ around $\bar z$ where all trajectories $z$ satisfying \eqref{eq:clphsys2} converge to the set $\mathcal S\subset \eqset_1$. \new{By continuing along the same lines as the proof of Theorem \ref{thm:grad}, convergence of each trajectory to a point is proven.} 
\end{proof}

\begin{myrem}
  Adding the quadratic term in the social welfare problem as done in \eqref{eq:modS} provides an additional advantage. As this introduces more damping in the resulting gradient-method-based controller, see \eqref{eq:adddamp}, it may improve the \blue{convergence properties} of the closed-loop dynamics \blue{\cite{boyd2011distributed}, \cite{rockafellar1973multiplieraugmented}}. Moreover, the amount of damping injected into the system depends on parameter $\rho$, which can be chosen freely. %Moreover, there is freedom in choosing the parameter $\rho>0$. 
\end{myrem}

\section{Conclusions and future research}
\label{sec:concl-future-rese}
In this paper a unifying and systematic energy-based approach in modeling and stability analysis of power networks has been established.  Convergence of the closed-loop system to the set of optimal points using gradient-method-based controllers have been proven using passivity based arguments. This result is extended to the case where nodal power constraints are included into the problem as well. However, for line congestion and power transmission cost the  power network is required to be acyclic to prove asymptotic stability to the set of optimal points. %For the case of linearized power flow, line congestion can be included without imposing this requirement on the network topology \cite{zhang1achieving}. 
%We have shown that not only gradient method based controllers admit a \pH\ description but also the DAPI controller and an internal-model-based controller  can be put into \pH\ form, which allows for a straightforward stability analysis of the closed-loop system. 

%We have shown that the \pH\ framework lends it self to the integration of real-time dynamics pricing algorithms.

The results established in this paper lend themselves to many possible extensions. One possibility is to design an additional (passive) controller that regulates the voltages to the desired values or achieves alternative objectives like (optimal) reactive power sharing. \blue{This could for example be realized by continuing along the lines of \cite{de2015modular}.} 

% \blue{
% While the main focus of this paper is on optimal \emph{active} power sharing, we stress that is also possible to consider (optimal) \emph{reactive} power sharing simultaneously, see e.g. \cite{de2015modular}, although this is beyond the scope of this paper. As a result, only active power is considered in the optimization problem \eqref{eq:minprobbasic}.  \magenta{(More vague, less specific, move to conclusion, keep referring to \cite{de2015modular}.)}}
% % One possibility is to look more carefully at the regulation of % the 
% time-varying voltages and how  reactive power sharing can be obtained.

Recent observations\blue{, see \cite{stegink2015optimal},} suggest that the \pH\ framework also lends itself to consider higher-dimensional models for the synchronous generator than the third-order model used in this paper, while the same controllers as designed in the present paper can be used in this case as well. \alert{Some of these results have been submitted to the upcoming IEEE Conference on Decision and Control 2016 \cite{stegink2015optimal}. }\blue{In addition, current research includes extending the results of the present paper to network-preserving models where a distinction is made between generator and load nodes.} %\magenta{First results look promising.}

One of the remaining open questions is how to deal with line congestion and power transmission costs in cyclic power networks with nonlinear power flows. In addition, all of the results established for the nonlinear power network only provide \emph{local} asymptotic stability to the set of optimal points. Future research includes determining the region of attraction.\alert{ in this kind of systems.}

\bibliographystyle{IEEEtran}
\bibliography{../../00Project/mybib} 

\ifCLASSOPTIONcaptionsoff
  \newpage
\fi

\begin{IEEEbiography}[{\includegraphics[width=1in,height=1.25in,clip,keepaspectratio]{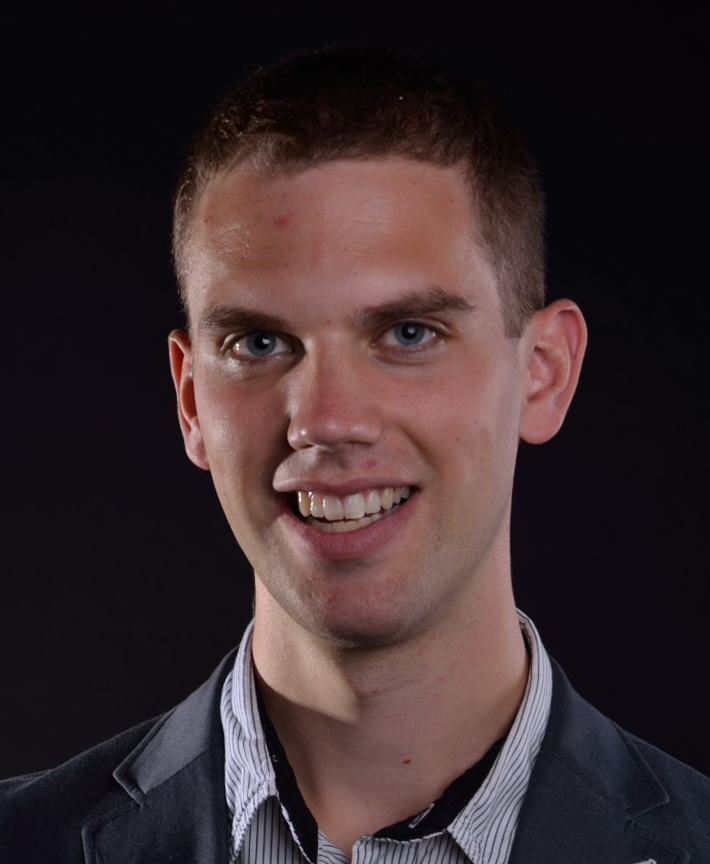}}]{Tjerk Stegink}
  is a Ph.D. candidate at the Engineering and Technology Institute,
  Faculty of Mathematics and Natural Sciences, University of
  Groningen, the Netherlands. He received his B.Sc. (2012) in Applied
  Mathematics and M.Sc. (2014, cum laude) in Systems, Control and
  Optimization from the same university. His main research interests
  are in modeling and nonlinear distributed control of power systems.
\end{IEEEbiography}

\begin{IEEEbiography}[{\includegraphics[width=1in,height=1.25in,clip,keepaspectratio]{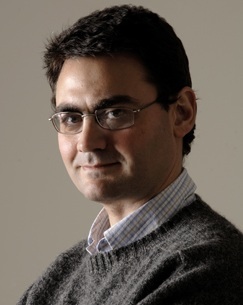}}]{Claudio De Persis}
  received the Laurea degree (cum laude) in electrical engineering in
  1996 and the Ph.D.  degree in system engineering in 2000, both from
  Sapienza University of Rome, Rome, Italy. He is currently a
  Professor at the Engineering and Technology Institute, Faculty of
  Mathematics and Natural Sciences, University of Groningen, the
  Netherlands. He is also affiliated with the Jan Willems Center for
  Systems and Control. Previously he was with the Department of
  Mechanical Automation and Mechatronics, University of Twente and
  with the Department of Computer, Control, and Management
  Engineering, Sapienza University of Rome. He was a Research
  Associate at the Department of Systems Science and Mathematics,
  Washington University, St.  Louis, MO, USA, in 2000–2001, and at the
  Department of Electrical Engineering, Yale University, New Haven,
  CT, USA, in 2001–2002. His main research interest is in control
  theory, and his recent research focuses on dynamical networks,
  cyberphysical systems, smart grids and resilient control. He was an
  Editor of the International Journal of Robust and Nonlinear Control
  (2006–2013), and is currently an Associate Editor of the IEEE
  Transactions On Control Systems Technology (2010–2015), of the
  IEEE Transactions On Automatic Control (2012–2015), and of
  Automatica (2013–present).
\end{IEEEbiography}

\begin{IEEEbiography}[{\includegraphics[width=1in,height=1.25in,clip,keepaspectratio]{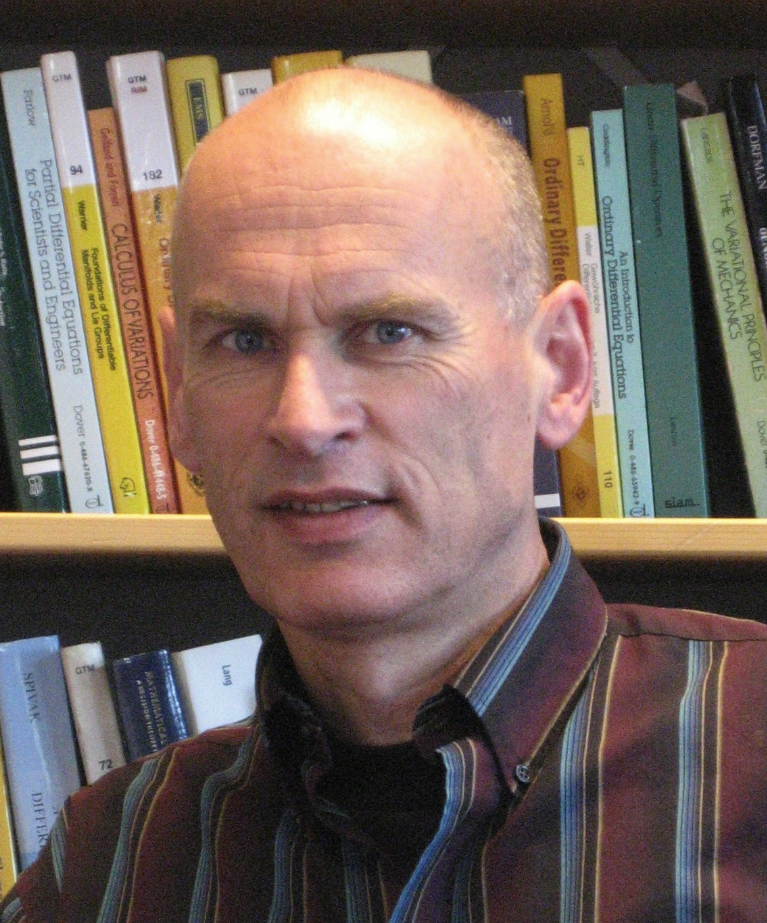}}]{Arjan van der Schaft}
  received the undergraduate (cum laude) and Ph.D. degrees in Mathematics from the University of Groningen, The Netherlands. In 1982 he joined the Department of Applied Mathematics, University of Twente, where he was appointed as full professor in Mathematical Systems and Control Theory in 2000.  In September 2005 he returned to his Alma Mater as a full professor in Mathematics.

  Arjan van der Schaft is Fellow of the Institute of Electrical and Electronics Engineers (IEEE), and Fellow of the International Federation of Automatic Control (IFAC). He was Invited Speaker at the International Congress of Mathematicians, Madrid, 2006. He was the 2013 recipient of the 3-yearly awarded Certificate of Excellent Achievements of the IFAC Technical Committee on Nonlinear Systems.

  He is (co-)author of the following books: System Theoretic Descriptions of Physical Systems (1984), Variational and Hamiltonian Control Systems (1987, with P.E. Crouch), Nonlinear Dynamical Control Systems (1990, with H. Nijmeijer), L2-Gain and Passivity Techniques in Nonlinear Control (1996, 2000), An Introduction to Hybrid Dynamical Systems (2000, with J.M. Schumacher), Port-Hamiltonian Systems Theory: An Introductory Overview (2014, with D. Jeltsema).
\end{IEEEbiography}

\end{document}